\documentclass[a4paper]{amsart}
\usepackage{cases}
\usepackage{amsfonts}
\usepackage{amsthm}
\usepackage{latexsym,bm}
\usepackage{amsmath,amssymb,amscd,bbm,amsthm,mathrsfs,dsfont}
\usepackage{mathrsfs}
\usepackage{pb-diagram}
\usepackage{color}
\usepackage{xypic}
\usepackage{mathrsfs}
\usepackage{indentfirst}
\usepackage{graphicx}
\usepackage{tikz}
\usetikzlibrary{decorations.markings,calc}
\newtheorem{mytheo}{Theorem}[section]

\newtheorem{cor}[mytheo]{Corollary}
\newtheorem{lem}[mytheo]{Lemma}
\newtheorem{prop}[mytheo]{Proposition}
\newtheorem{defn}[mytheo]{Definition}
\newtheorem{rem}[mytheo]{Remark}

\newtheorem{eg}[mytheo]{Example}

 \normalbaselines
\begin{document}
\title[The geometric realization of regular path complexes via (co)homology]
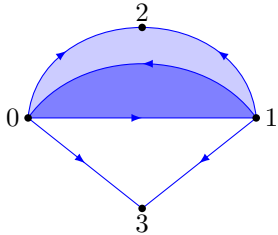
{The geometric realization of regular path complexes\\ via (co)homology }
\author{Fang Li $\;\;\;\;\;\;$ Bin Yu$^*$ $\;\;\;\;\;\;$}
\address{Fang Li
\newline School of Mathematical Sciences,
Zhejiang  University, Hangzhou 310027, P.R.China}
\email{fangli@zju.edu.cn}
\address{Bin Yu
\newline Department of Mathematics,
China Jiliang University, Hangzhou 310018, P.R.China}
\email{binyu@cjlu.edu.cn}
\date{\today}
\maketitle

\renewcommand{\thefootnote}{\alph{footnote}}
\setcounter{footnote}{-1} 
\renewcommand{\thefootnote}{\alph{footnote}}
\setcounter{footnote}{-1} 
\setcounter{footnote}{-1} \footnote{*\emph{Corresponding author} }

\begin{abstract}

The aim of this paper is to give the geometric realization of regular path complexes via (co)homology groups with coefficients in a ring $R$.
Concretely, for each regular path complex $P$, we associate it with a singular $\Delta$-complex $S(P)$ and show that the (co)homology
groups of $P$ are isomorphic to those of $S(P)$ with  coefficients in $R$. As a direct result we recognize path (co)homology as Hochschild (co)homology
in case that $R$ is commutative and $P$ regular finite. Analogues of the Eilenberg-Zilber theorem and K\"{u}nneth formula are also showed for the Cartesian product
and the join of two regular path complexes. In fact, we meanwhile improve some previous results which are covered by these conclusions in this paper.
\end{abstract}

\tableofcontents

\section{Introduction}

\noindent S.-T.Yau et al. (\cite{Lin2}) introduced the concept of quantum tunneling on graphs. As a preparation of their work, Yau and his
collaborators had previously studied the eigenfunction defined for graphs (\cite{Cheng,Lin1}), which is a generalization of the version for
smooth Riemannian manifolds. Using this, one can study the nodes of a graph and to control the volume growth of a graph via estimate
of some variables. It turns out to have a great amount of practical applications. Such theories  of graphs are developed parallel to that of
Riemannian geometry in many ways.

When considering digraphs (i.e., quivers), we can define their (co)homology groups in a similar way as we have done in topology (see \cite{Grig1}).
These (co)homology groups can be used directly to study the relationship among some aspects of digraphs and their related theories.
In this manner, it is possible to give a geometric interpretation of eigenfunction and quantum tunneling on digraphs once we obtain the
geometric realization of digraphs given via the (co)homology groups of the associated path complexes (please see Definition 2.2). This is the motivation of our work.

In this paper, we introduce and study the path (co)homology of a path complex with coefficients in a unital ring $R$. The central
idea is based on a topological approach. To be short, given any regular path complex $P$, the geometric realization $S(P)$ of $P$ as a
singular $\Delta$-complex (see Definition \ref{sdelta}) is obtained from a map $F_\Delta$ defined by an inductive process. Then for any unital ring $R$,
one can recognize path (co)homology of $P$ as simplicial (or equivalently cellular) (co)homology of $S(P)$ with coefficients in $R$. Concretely, we have the following:
\\
\\
{\bf{Main Theorem}} (Theorems \ref{thm3.10} and \ref{thm3.11})\;
{\em Let $P$ be a regular path complex and $R$ be a ring, the map $F_\Delta$ induces isomorphisms $H_\ast(P)\cong H_\ast(S(P))$ and $H^\ast(P)\cong H^\ast(S(P))$.}
\\
\\
Historically, in order to study the topological structure of directed graphs (digraphs) and further to classify them, there are many attempts
to form a homological theory for digraphs. Among these approaches there are three of them being well-known: regarding a digraph as a special
one-dimensional simplicial complex, considering all the cliques of a digraph as simplices of the corresponding dimensions (\cite{Chen,Ivan}),
or taking Hochschild cohomology of the path algebra of a digraph (\cite{Happ}). But as it is commented in \cite{Grig1} that all these
approaches have their emphasises and limitations. In view of this, the authors of \cite{Grig1} introduced the notions of path complexes and
path homology over a field (while their cohomology version can be found in \cite{Grig5}). This new (co)homology theory for digraphs, including
its sequel notions and results, not only shares many properties with the above approaches but also avoid many limitations. It is shown that one
can use path homology to give a refined classification of digraphs via some homological invariants such as the dimensions of the homology
groups, Euler characteristic and so on. On the other hand, as we can see that from \cite{Grig3}, path cohomology theory is a powerful tool when
one deals with the algebraic aspect of simplicial cohomology, in fact it allows a delicate proof to the isomorphism obtained in \cite{Gers} without
using Cohomology Comparison Theorem.

Furthermore, the classic Eilenberg-Zilber theorem and K\"{u}nneth formula holding for the Cartesian product space of two topological spaces
have also analogues in the theory of path homology. In fact, similar results hold for both of the Cartesian product and the join of two path
complexes over a field (\cite{Grig1}).

Meanwhile, as we usually do in the theory of simplicial (co)homology, it seems that there is no need to confine the coefficients of path
(co)homology in a field, since different coefficient rings usually induce different (co)homology groups. For instance, when one ignores the
orientation of a simplicial complex one should consider directly its simplicial homology groups with coefficients in $\mathbb{Z}_2$ instead of
$\mathbb{Z}$.

So in this paper, we define and study the path (co)homology  with coefficients in a general ring $R$. Our main goal is to
recognize the path (co)homology as a geometric (co)homology theory. This would provide a new approach to study graph theory and give
many of the results in \cite{Grig1, Grig2, Grig3} an intuitive interpretation. To be short, this is done by the geometric realization of path
complexes as we mentioned at the beginning, the motivation is to view a path complex as a simplicial set. Moreover, given a finite path complex $P$
over a commutative ring $R$, by this result and the results of \cite{Gers, Grig3} one may also construct an associative algebra $A_{S(P)}$ over
$R$ associated to $P$, which is very different from the related path algebra even for $P$ arising from a digraph, and by what we have in hand
it is not surprising that one could prove that the Hochschild (co)homology of $A_{S(P)}$ is isomorphic to path (co)homology of $P$.

The paper is organized as follows. We set off after reviewing some definitions and notations in Section 2. In Section 3, we establish the
correspondence of regular path complexes and singular $\Delta$-complexes which share the isomorphic (co)homology groups with
coefficients in general rings (see Theorems \ref{thm3.10} and \ref{thm3.11}). This recognizes path (co)homology as simplicial (co)homology,
and as we have pointed out that, it allows us to simplify proofs of many results in \cite{Grig1} and \cite{Grig2}. As a first application, in Section 4,
we further recognize path (co)homology with coefficients in a commutative ring as Hochschild (co)homology (see Theorems 4.1 and 4.3) for regular finite path complexes.
Section 5 is dedicating to further applications of the obtained results. Analogues of the Eilenberg-Zilber theorem are obtained in a unified way for the
Cartesian product and the join of two regular path complexes (Theorems 5.8 and 5.11) over any commutative rings. These imply respectively two
general K\"{u}nneth formulae for path homology with coefficients in principle ideal domains. Note that not only this generalizes the previous
result in \cite{Grig1} obtained for path complexes over a field $K$, but also our proofs here go rather different from those given in
\cite{Grig1}.

\section{Preliminaries}
\noindent Throughout this paper, $K$ denotes a field and $R$ denotes a unital ring if not specified.
We recall from \cite{Grig1} some notations and definitions in this section, though most of them are
defined temporarily in the case where $K$ is a field, as we shall see that they can be easily extended
to the case when one replaces $K$ by a unital ring $R$.

\subsection{Path complexes}

\begin{defn}
\emph{Let $V$ be an arbitrary non-empty finite set whose elements will be called \emph{vertices}.
For any non-negative integer $p$, an \emph{elementary $p$-path} on a set $V$ is any ordered
sequence $\{i_{k}\}^{p}_{k=0}$ (or simply written as $i_{0}\cdots i_{p}$) of $p+1$ vertices
(needs not be distinct) of $V$.
 Furthermore, an elementary path $i_{0}\cdots i_{p}$ is said to
be \emph{non-regular} if $i_{k-1}=i_{k}$ for some $1\leq k\leq p$, and \emph{regular} otherwise.}
\end{defn}

Denote by $\Lambda_{p}=\Lambda_{p}(V; K)$ the $K$-linear space that consists of all formal linear
combinations of all elementary $p$-paths with the coefficients from $K$. The elements of $\Lambda_{p}$
are called \emph{$p$-paths} on $V$, and an elementary $p$-path $i_{0}\cdots i_{p}$ as an element
in $\Lambda_{p}$ is written as $e_{i_{0}\cdots i_{p}}$. Obviously the basis in $\Lambda_{p}$ is the
family of all elementary $p$-paths, and each element $v$ in $\Lambda_{p}$ has the following form:
$$v=\sum_{i_{0},\cdots,i_{p}\in V}v^{i_{0}\cdots i_{p}}e_{i_{0}\cdots i_{p}}$$
where $v^{i_{0}\cdots i_{p}}\in K$. For any $p\geq -1$, consider the subspace of $\Lambda_{p}$
spanned by the regular elementary paths: $\mathcal{R}_{p}=\mathcal{R}_{p}(V; K):=$span$\{e_{i_{0}\cdots i_{p}}$:
$i_{0}\cdots i_{p}$ is regular\}, whose elements are called \emph{regular $p$-paths}.

For any $p\geq0$, define the boundary operator $\partial: \Lambda_{p}\rightarrow\Lambda_{p-1}$
as a linear operator that acts on elementary paths by
\[
\partial e_{i_{0}\cdots i_{p}}=\sum_{q=0}^{p}(-1)^{q}e_{i_{0}\cdots \widehat{i_{q}}\cdots i_{p}} \tag{2.1}\label{pdif}
\]
where the hat $\widehat{i_{q}}$ means omission of the index $i_{q}$. Note that such boundary
operators make $\Lambda_{\ast}=\{\Lambda_{p}\}$ a chain complex (see \cite[Lemma 2.4]{Grig1}).
Similarly we can define the regular complex $\mathcal{R}_{\ast}=\{\mathcal{R}_{p}\}$ consisting
of regular elements and with natural boundary operators, i.e., those boundary operators are defined
by the induced maps of $\partial$ acting on the quotient space $\Lambda_{p}$ over non-regular
paths, and it is easy to check that $\mathcal{R}_{\ast}=\{\mathcal{R}_{p}\}$ is a chain complex
under such boundary operators (see \cite{Grig1} for details). Let $V$, $V^{\prime}$ be two finite sets, by
definition, any map $f: V\rightarrow V^{\prime}$ gives rise to two natural morphisms
$\Lambda_{\ast}(V)\rightarrow\Lambda_{\ast}(V^{\prime})$ and $\mathcal{R}_{\ast}(V)\rightarrow\mathcal{R}_{\ast}(V^{\prime})$.

The central concept in our study is the following.

\begin{defn}\label{pathcomplex}
\emph{A \emph{path complex} over a finite set $V$ is a non-empty collection $P$ of elementary paths
on $V$ with the following property: for any $n\geq0$, if $i_{0}\cdots i_{n}\in P$ then also
the truncated paths $i_{0}\cdots i_{n-1}$ and $i_{1}\cdots i_{n}$ belong to $P$. The elementary
$n$-paths from $P$ is denoted by $P_{n}$. If all the paths in $P$ are regular, then $P$ is
called a \emph{regular} path complex. $P$ is called \emph{finite} if  $P_{\geq m}$ are all empty for some $m>0$.}
\par
\end{defn}

Here is an example of path complex:
\begin{eg} \label{Exam}
\emph{Let $V=\{0, 1, ..., 8\}$, and $P$ be a path complex in which the elementary paths are give by:\\
\indent $0$-paths: $0, 1, ..., 8$\\
\indent $1$-paths: $01, 02, 03, 04, 05, 07, 08, 12, 34, 35, 45, 67, 68, 78$\\
\indent $2$-paths: $012, 034, 035, 045, 345, 078, 678$\\
\indent $3$-paths: $0345$.\\
In fact, given a digraph $G$, there is a natural way to associate it
with a path complex $P(G)$ whose vertices and elementary paths are decided by the
digraph in the obvious way. As one can easily check that, the associated path complex of the
following digraph is exactly $P$ (for more details please see \cite[Example 3.3]{Grig1})}.
\end{eg}
\[
\begin{tikzpicture}[scale=1.5,decoration={
    markings,
    mark=at position 0.5 with {\arrow{latex}}}
    ]
     \node[circle,fill=black,inner sep=1pt] at (0,0)(0){} node at +(-90:0.2){0};
     \node[circle,fill=black,inner sep=1pt] at (-1,-0.7)(1){} node[shift=(-135:0.2)] at (1) {1};
     \node[circle,fill=black,inner sep=1pt] at (0.5,-0.7)(2){} node[shift=(-45:0.2)] at (2) {2};
     \node[circle,fill=black,inner sep=1pt] at (140:1.15)(3){} node[shift=(-90:0.2)] at (3) {3};
     \node[circle,fill=black,inner sep=1pt] at (170:2)(4){} node[shift=(180:0.2)] at (4) {4};
     \node[circle,fill=black,inner sep=1pt] at (110:2)(5){} node[shift=(90:0.2)] at (5) {5};
     \node[circle,fill=black,inner sep=1pt] at (0.5,0.5)(6){} node[shift=(180:0.2)] at (6) {6};
     \node[circle,fill=black,inner sep=1pt] at (1,0)(7){} node[shift=(0:0.2)] at (7) {7};
     \node[circle,fill=black,inner sep=1pt] at (0.5,1.5)(8){} node[shift=(90:0.2)] at (8) {8};
    \draw[postaction={decorate}] (0)--(1);
    \draw[postaction={decorate}] (1)--(2);
    \draw[postaction={decorate}] (0)--(2);
    \draw[postaction={decorate}] (0)--(3);
    \draw[postaction={decorate}] (0)--(4);
    \draw[postaction={decorate}] (3)--(4);
    \draw[postaction={decorate}] (3)--(5);
    \draw[postaction={decorate}] (4)--(5);
    \draw[postaction={decorate}] (0)--(5);
    \draw[postaction={decorate}] (0)--(7);
    \draw[postaction={decorate}] (0)--(8);
    \draw[postaction={decorate}] (6)--(8);
    \draw[postaction={decorate}] (6)--(7);
    \draw[postaction={decorate}] (7)--(8);
\end{tikzpicture}
\]

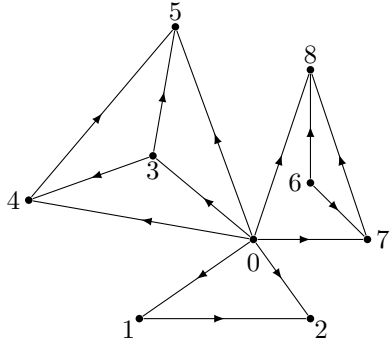
\centerline{Figure 1:\ \ \ \ A digraph $G$ with $P(G)=P$.}

\subsection{Path (co)homology}
\indent

\noindent When a path complex $P$ is fixed, all the $n$-paths of the form $\sum_{j=1}^{s}r_{j}e_{\mathbf{i}^{(j)}}$ with $s$
a finite integer, each $r_{j}\in K$ and $\mathbf{i}^{(j)}=i^{(j)}_0i^{(j)}_1\cdots i^{(j)}_n$ such that $e_{\mathbf{i}^{(j)}}\in P_{n}$ are called \emph{allowed}, otherwise are called \emph{non-allowed}. The set of all
allowed $n$-paths is denoted as $\mathcal{A}_{n}(P)=\mathcal{A}_{n}(P; K)$. Furthermore, for any $n\geq0$ we
define $\Omega_{n}(P)$ as follows:
\[
\Omega_{n}(P)=\Omega_{n}(P; K):=\{p_{n}|p_{n}\in\mathcal{A}_{n}(P)\ \  \mbox{and}\ \  \partial(p_{n})\in\mathcal{A}_{n-1}(P).\}.
\]
Apparently each $\Omega_{n}(P)$ is a $K$-module, namely a vector space over $K$. It is easy to verify that $\partial(\Omega_m(P))\subseteq\Omega_{m-1}(P)$
and $\partial^2=0$, thus we obtain a chain complex of $K$-modules:
\[
\Omega_\ast(P)=\Omega_\ast(P; K):=\cdots\rightarrow\Omega_n(P)\rightarrow\Omega_{n-1}(P)\rightarrow\cdots\rightarrow\Omega_0(P)\rightarrow0.
\]
Therefore, for any $n\geq 0$ we define the $n$-th \emph{path homology group} of $P$ as $H_n(\Omega_\ast(P))$, or denoted shortly by $H_n(P)$.
If the path complex $P$ is regular, which is the case we shall study in this paper, all the above definitions and notations have modified
versions when one replaces the boundary operator by the modified boundary operator which is used to define $\mathcal{R}_\ast$.

The above definitions and notations also have dual versions. For any integer $p\geq -1$, denote by $\Lambda^{p}=\Lambda^{p}(V; K)$ the linear
space of all $K$-valued functions on $(p+1)$-multiplicative product $V^{p+1}$ of set $V$. Otherwise we set $\Lambda^{\leq -2}=\{0\}$. In
particular, $\Lambda^0$ is the linear space of all $K$-valued functions on $V$, and $\Lambda^{-1}$ is the space of all $K$-values functions on
$\Lambda^{0}:=\{0\}$, that is, one can identify $\Lambda^{-1}$ with $K$. The elements of $\Lambda^{p}$ are called \emph{$p$-forms} on $V$, one
can identify $\Lambda^{p}$ with the dual space of $\Lambda_{p}$ via the canonical identity $\Lambda^{p}\cong\mbox{Hom}_{K}(\Lambda_{p}, K)$.
The boundary operator (\ref{pdif}) should be replaced now by exterior differential $d: \Lambda^{p}\rightarrow\Lambda^{p+1}$ given by
\[
(d\omega)_{i_{0}\cdots i_{p+1}}=\sum_{q=0}^{p+1}(-1)^{q}\omega_{i_{0}\cdots \widehat{i_{q}}\cdots i_{p+1}} \tag{2.2}\label{fdif}
\]
for any $\omega\in\Lambda^{p}$. Similarly we define the space of \emph{regular} $p$-forms
$\mathcal{R}^{p}=\mathcal{R}^{p}(V):=\mbox{Hom}_{K}(\mathcal{R}_{p}, K)$ (hereafter this means, any element in $\mathcal{R}^{p}$ always takes
$\Lambda_p\setminus\mathcal{R}_p$, i.e., non-regular $p$-paths to 0). Given a path complex $P$, we define the space of \emph{allowed} $p$-forms
$\mathcal{A}^{p}(P)=\mathcal{A}^{p}(P; K):=\mbox{Hom}_{K}(\mathcal{A}_{p}(P), K)$, also denote $$\mathcal{N}^{p}=\Lambda^p\setminus\mathcal{A}^{p}(P)\ \ \  \mbox{and}\ \ \ \mathcal{J}^{p}=\mathcal{N}^{p}+d\mathcal{N}^{p-1},$$
and define
\[
\Omega^p(P)=\mathcal{A}^p/(\mathcal{A}^p\cap\mathcal{J}^p).
\]
Actually, it follows from \cite[Lemma 3.19]{Grig1} that $\Omega^p(P)$ is the dual
space of $\Omega_p(P)$ while $d$ is dual to $\partial$, that is to say,
one has
\[
\Omega^p(P)\cong\mbox{Hom}_K(\Omega_p(P), K)\ \ \ \mbox{and}\ \ \ d\cong\mbox{Hom}_K(\partial, K) \tag{2.3}\label{fome}.
\]
It can be shown that $\{\Omega^p(P)\}$ amounts to a cochain complex with the
differential operator given by (\ref{fdif}), whereas the $n$-th \emph{path cohomology group}
of $P$ for any $n\geq 0$ is referred to the $n$-th cohomology group
$H^n(\Omega^\ast(P))$ of this cochain complex, which is denoted shortly by $H^n(P)$.

Hitherto, all the definitions are defined only for the case
where $K$ is a field. But it is not hard to see that all of them can be
easily carried over to the case when replacing $K$ by any ring $R$. To do this
there is no need to change a word but replacing all $K$-vector spaces
($K$-modules) by $R$-modules. In this paper, we shall focus on this more
general situation, that is, path (co)homology is understood to be with
coefficient in an associative unital ring $R$, and we shall omit ``$R$" in the
notation since there is no ambiguity.

\section{The geometric realization of regular path complexes}
\noindent In the rest of this paper, all path complexes are understood to be {\em regular} unless specified otherwise, and we shall give the geometric
realization of regular path complexes in this section, i.e., we will construct the correspondence of regular path complexes and
$\Delta$-complexes (see Definition 3.1 below).

Suppose we are given a path complex $P=\{P_{n}\}$ over a finite set $V=\{i_{j}\}$. The aim of this section is to construct a singular $\Delta$-complex
$S(P)$ (see below for definition) with desired homological property.
To do this we plan to associate each proper elementary $n$-path with an $n$-cell
$\Delta^{n}$ whose distinguished characteristic
map $\Delta^{n}\rightarrow S(P)$ will be given, and this process is demonstrated in several steps as follows.

\subsection{Construction of the singular $\Delta$-complexes}
\indent

\noindent Recall that a standard $n$-simplex $\Delta^{n}$ is an
$n$-dimensional convex polyhedron in $\mathbb{R}^{n}$ containing $n$
points which are the $n$ standard basis vectors for $\mathbb{R}^{n}$.
We enumerate the $n+1$ vertices in order, say, $0, 1, \cdots, n$, and give each
edge with two vertices $k,l$ $(0\leq k<l\leq n)$ the direction from $k$ to $l$.
Similarly for any $k$-faces of $\Delta^{n}$, we assign to it a standard orientation in a
well-known way (see also, for example, \cite[p.233]{Hatc}). That is, any other
ordering of the vertices obtained from an even permutation of the
original ordering ($0<1<\cdots<n$) is viewed as the same orientation of
$\Delta^{n}$, which is said to be a \emph{canonical orientation} of
$\Delta^{n}$, otherwise we say that the ordering of vertices gives an opposite orientation to the canonical one.
The same definition can be easily extended to any $\Delta$-complexes by characteristic maps, whose definition is given as follows:

\begin{defn}[\cite{Hatc}]
\emph{
A $\Delta$-complex structure on a space $X$ is a collection of maps $\sigma_\alpha: \Delta^{n}\rightarrow X$, with $n$ depending on the index $\alpha$, such that:\\
{\rm(i)} The restriction $\sigma_\alpha$ on the interior of $\Delta^{n}$ is injective, and each point of $X$ is in the image of exactly one such restriction of $\sigma_\alpha$.\\
{\rm(ii)} Each restriction of $\sigma_\alpha$ to a face of $\Delta^{n}$ is one of the maps $\sigma_\beta: \Delta^{n-1}\rightarrow X$. Here we are identifying the face
of $\Delta^{n}$ with $\Delta^{n-1}$ by the canonical linear homeomorphism between them that preserves the ordering of the vertices.\\
{\rm(iii)} A set $A\subset X$ is open iff $\sigma_\alpha^{-1}(A)$ is open in $\Delta^{n}$ for each $\sigma_\alpha$.
}
\end{defn}

Historically, the above definition was first introduced by Eilenberg and Zilber under the name ``semi-simplicial complexes" as a compromise between simplicial sets
and  simplicial complexes.  A $\Delta$-complex can be also viewed as a CW complex $X$ in which each $n$-cell $e_\alpha^n$ is provided
with a distinguished characteristic map $\delta_\alpha : \Delta^n\rightarrow X$ (which is a homeomorphism from the interior of   $\Delta^n$ onto $X$)
such that the restriction of $\delta_\alpha$ to each face  $\Delta^{n-1}$ of $\Delta^{n}$ is the distinguished $\delta_\beta$ for some $(n-1)$-cell $e_\beta^{n-1}$ .

But in view of that the restriction $\sigma_\alpha$ to a face of $\Delta^{n}$ dose not allow any degeneracy
for $\Delta$-complexes,  for our purpose, a more generalized conception is needed here, which can also be found in \cite{Hatc}.

\begin{defn}\label{sdelta}
\emph{
A \emph{singular $\Delta$-complex,} or \emph{s$\Delta$-complex,} is a CW complex $X$ with distinguished characteristic maps $\delta_\alpha : \Delta^n\rightarrow X$
whose restrictions to faces are the compositions $\delta_\beta q: \Delta^{n-1}\rightarrow\Delta^k\rightarrow X$ for $q$ a linear surjection taking vertices to vertices,
preserving order. Simplicial maps between s$\Delta$-complexes are defined just as for $\Delta$-complexes.
}
\end{defn}

For convenient we shall use simplicial maps instead of cellular maps to specify the attachments all the time (as that we have done
in the above definition), and we do not distinguish
between the simplicial (co)homology and cellular (co)homology of a s$\Delta$-complex in the following sections since they are canonically isomorphic. Note
that one can also take products of s$\Delta$-complexes by the same subdivision procedure as for $\Delta$-complexes, we shall
use this fact in Section 5. To avoid ambiguity, we introduce the following definition.

\begin{defn}
\emph{An $n$-simplex with ordered vertices which give out a canonical orientation is said to be an \emph{ordered} $n$-simplex, and a s$\Delta$-complex is said to be \emph{ordered}
if it consists only of the images of orderded simplices together with the collection of maps preserving the ordering of vertices.}
\end{defn}

In the sequel, we assume that all $n$-simplices (hence s$\Delta$-complexes) are ordered if there is no special statement. With
all these preparation in hand, next we shall bring out the
construction of the s$\Delta$-complexes as the geometric realization of given path complexes in two steps, and to do this one just need to give
all the distinguished characteristic maps.

Suppose now $V=\{i_j\}$ is a finite set, and $P$ is a regular path complex defined over $V$. The following definition is crucial in our construction.

\begin{defn}\label{value}
\emph{
The set $V$ gives rise to a simplicial set $X(V)$ whose set of $n$-simplices are given by all elementary $n$-paths on $V$ in the following manner.
Let $n\geq 1$ be an integer and $[n]=\{0, 1,\cdots, n\}$ be an ordered vertices set. Then each elementary $n$-path $e_n=i_0i_1\cdots i_n$ on the set $V$ induces a map $X_{e_n}$ that sends $[n]$ to $\{i_0, i_1, \cdots, i_n\}\in X_n=V^{n+1}$ by preserving orders, each such map associates any nondecreasing map $f: [m]\rightarrow[n]$ (particularly the face maps $\varepsilon_j$ and degeneracy maps $\eta_j$) with a map $X_{e_n}(f): X_n\rightarrow X_m$ given by the composition $X_{e_n}f$.
Thus each $e_n$ defines
an $n$-simplex $X_{e_n}$, and all these $n$-simplices amount to the desired simplicial set $X(V)$.
}
\end{defn}

Now we need to take the geometrical realization $|X(V)|$ of $X(V)$ for the first step, and for simplicity, we denote it by $S(V)$. It is known that $S(V)$ could be obtained as an s$\Delta$-complex. But in view of clearness, here we shall give a complete discription of $S(V)$ as follows, and it is convenient to introduce the following definition before this.

\begin{defn}\label{reduce}
\emph{
Let $e_n=i_0\cdots i_n$ be an elementary $n$-path over $V$ for some integer $n>1$. If $e_n$ is {\em non-regular}, that is, there
exists at least a pair of two neighbouring  vertices $i_j$ and $i_{j+1}$ with $0\leq j<n$ such that $i_j=i_{j+1}$, then we define an
{\em elementary $(n-1)$-path $e_{n-1}$ from $e_n$} by the formula $e_{n-1}=i_0\cdots \widehat{i_j}\cdots i_n$. Moreover,
if $e_{n-1}$ is again non-regular, one can continue this procedure to obtain finally a regular elementary $m$-path $e_m$ for
some $1\leq m<n-1$. Such an $m$-path $e_m$ is called the \emph{reduced path} of $e_n$.
}
\end{defn}

The construction of $S(V)$ is hence can be done by induction on its $n$-skeleton.

(i)\; \emph{$0$-skeleton of $S(V)$:} Each elementary 0-path, say, $e_0=i_j$ gives naturally a characteristic map $\phi_{e_0}$ on $\Delta^0$ by the map $X_{e_0}$
from $\{0\}$ to $\{i_j\}$.

(ii)\; \emph{$1$-skeleton of $S(V)$:} Each regular elementary 1-path, say, $e_1=i_ki_l$  gives a characteristic map $\phi_{e_1}$ on $\Delta^1$ by the map $X_{e_1}:$
$\{0,1\}\rightarrow \{i_k, i_l\}$ in the following way: the vertices $0$ and $1$ of $\Delta^1$, by step (i), are
mapped to the vertices $i_k$ and $i_l$  respectively, and $\Delta^1$  is mapped to an edge from $i_k$ to $i_l$.  Otherwise if
$e_1$ is non-regular, then by Definition \ref{reduce}, one obtains its reduced path $e_0=i_k$.
If we define $\phi_{e_1}:=\phi_{e_{0}}$, where $\phi_{e_0}$ is already defined
for $e_0$  by inductive step (i), then such a reduced path $e_0$ assigns $e_1$ to a 0-cell with the  characteristic map $\phi_{e_1}$.

(iv)\; \emph{$n$-skeleton of $S(V)$:} Suppose for all $s$ with $1\leq s<n$ the characteristic maps of any elementary $s$-paths are given,
the discussion is divided into two cases. (1) Each regular elementary $n$-path, say, $e_n=i_0\cdots i_n$ gives a characteristic map $\phi_{e_n}$ on $\Delta^n$
in the following way. Since there exists a canonically correspondence between
the $j^{th}$-faces of $\Delta^n$ and the ordered vertices sequences $\varepsilon_j([n-1])=\{0,1,\cdots,\widehat{j},\cdots, n\}$ for $0\leq j\leq n$, one can associate such $j^{th}$-faces with respectively elementary $(n-1)$-paths $i_0i_1\cdots\widehat{i_j}\cdots i_n$, which are exactly the images of the maps $X_{e_n}\varepsilon_j:[n-1]\rightarrow \{i_0, i_1, \cdots, i_n\}$
by preserving orders. If one of such $(n-1)$-paths is regular, then it gives a characteristic map on the corresponding $j^{th}$-face by inductive
steps. If one of the  $(n-1)$-paths, say, $e_{n-1}=i_0i_1\cdots\widehat{i_j}\cdots i_n$ is non-regular for some $0\leq j\leq n$, consider each of its
reduced paths, say an elementary $m$-path $e_{m}=i_0i_1\cdots\widehat{i_k}\cdots\widehat{i_j}\cdots\widehat{i_l}\cdots i_n$ obtained from
$e_{n-1}$ by a further omission of $(n-m-1)$ vertices for some $m<n-1$. It induces naturally a projection, denoted as $q_{e_{n-1}}$, from the
$(n-1)$-face represented by $\{0,1,\cdots,\widehat{j},\cdots,n\}$ onto some of its $m$-face represented by
$\{0,1,\cdots,\widehat{k},\cdots,\widehat{j},\cdots,\widehat{l},\cdots,n\}$. Combining the projection with the characteristic
map $\phi_{e_{m}}$ obtained from $e_{m}$ by inductive process, one finally gets a composition $\phi_{e_{n-1}}:=\phi_{e_{m}}q_{e_{n-1}}$ acting
on the original $j^{th}$-face of $\Delta^n$ which is represented by ordered vertices sequences $\{0,1,\cdots,\widehat{j},\cdots,n\}$. Hence the desired
characteristic map  $\phi_{e_n}$ on the $(n-1)$-faces of $\Delta^n$ can be defined by these maps. To conclude, one demands that the characteristic
map $\phi_{e_n}$ confined on the interior of $\Delta^n$ is a homeomorphism (see the following Example). (2) Otherwise if $e_n$ is non-regular,
then its reduced path $e_m$ defines a characteristic map $\phi_{e_n}:=\phi_{e_m}$ on $\Delta^{m}$ for some $m<n$, which are defined by inductive steps.
Notice that whether $e_n$ is regular or not, the characteristic map $\phi_{e_n}$ here defined is uniquely determined up to homotopy.

In this manner, one obtains finally an s$\Delta$-complex $S(V)$ from $V$.

\begin{eg}\label{eg3.3}
\emph{Let $V=\{a, b, c\}$ be a set of three vertices, the following figures demonstrate four distinguished characteristic maps associated with
respectively elementary paths $e_{ab}$, $e_{abc}$, $e_{aba}$ and $e_{abab}$. As one can see, the resulting s$\Delta$-complexes are
respectively a 1-cell, two 2-cells and a 3-cell. In Figure 4, regarding to the non-regular elementary $2$-path $e_{aa}$ occurs in the boundary
of $e_{aba}$, the characteristic map $\phi_{e_{aba}}$ gives degeneracy on $1$-face ``$02$", i.e. the composition of linear projection map
from the edge ``$02$" onto its vertex $0$ with the map $\phi_{e_0}$, and $\phi_{e_{aba}}$ is homeomorphism on the interior of $\Delta^2$,
which is given by for example, the function $f(x, y)=xy$, assuming the coordinates of vertices $0,1$ and $2$ are respectively $(0, -1), (2, 1)$
and $(0, 1)$. Similarly, in Figure 5, the boundary map of $\phi_{e_{abab}}$ sends the two  $2$-faces ``013" and ``023" associated respectively
with the boundary paths $e_{aab}$ and $e_{abb}$ onto one common edge from $a$ to $b$, while $\phi_{e_{abab}}$ gives a homeomorphism
on the interior of $\Delta^3$.}
\end{eg}

\[
\begin{tikzpicture}[scale=1,baseline=(current bounding box.center)]
\begin{scope}[decoration={
    markings,
    mark=at position 0.5 with {\arrow{latex}}}
    ]

    \draw[blue,postaction={decorate}] (180:2)--(0,0);
     \node[circle,fill=black,inner sep=1pt] at (0,0)(b){} node at +(0:0.2){$1$};
     \node[circle,fill=black,inner sep=1pt] at (180:2)(a1){} node[shift=(180:0.2)] at (a1) {$0$};
\end{scope}
\end{tikzpicture}
\hspace{0.5cm}
\Longrightarrow
\hspace{0.5cm}
\begin{tikzpicture}[scale=1,baseline=(current bounding box.center)]
\begin{scope}[decoration={
    markings,
    mark=at position 0.5 with {\arrow{latex}}}
    ]

    \draw[blue,postaction={decorate}] (180:2)--(0,0);
    \node[circle,fill=black,inner sep=1pt] at (180:2)(a){} node[shift=(180:0.2)] at (a) {$a$};
    \node[circle,fill=black,inner sep=1pt] at (0,0)(b){} node at +(0:0.2){$b$};
\end{scope}
\end{tikzpicture}
\]
\vspace{5pt}
\begin{center}{Figure 2 :\ \ \ \ $\phi_{e_{ab}}(\Delta^{1})$ is obtained from  $\Delta^1$ with the vertices decided by the valuation of $e_{ab}$.}
\end{center}

\[
\begin{tikzpicture}[scale=1,baseline=(current bounding box.center)]
\begin{scope}[decoration={
    markings,
    mark=at position 0.5 with {\arrow{latex}}}
    ]
    \draw[fill=blue!20, line join=round] (180:2)--(0,0)--(120:2) -- cycle;
    \draw[blue,postaction={decorate}] (180:2)--(0,0);
    \draw[blue,postaction={decorate}] (0,0)--(120:2);
    \draw[blue,postaction={decorate}] (180:2)--(120:2);
     \node[circle,fill=black,inner sep=1pt] at (0,0)(b){} node at +(0:0.2){$1$};
     \node[circle,fill=black,inner sep=1pt] at (180:2)(a1){} node[shift=(180:0.2)] at (a1) {$0$};
     \node[circle,fill=black,inner sep=1pt] at (120:2)(a2){} node[shift=(90:0.2)] at (a2) {$2$};
\end{scope}
\end{tikzpicture}
\hspace{0.5cm}
\Longrightarrow
\hspace{0.5cm}
\begin{tikzpicture}[scale=1,baseline=(current bounding box.center)]
\begin{scope}[decoration={
    markings,
    mark=at position 0.5 with {\arrow{latex}}}
    ]
    \fill[blue!20] (-1,0) to[out=-85,in=180] (0,-0.6) to[out=0,in=-95] (1,0) to[out=95,in=0] (0,0.6) to[out=180,in=85] (-1,0) -- cycle;
    \draw[color=blue,postaction={decorate}] (-1,0) to[out=-85,in=180] (0,-0.6) to[out=0,in=-95] (1,0);
    \draw[color=blue,postaction={decorate}] (1,0) to[out=95,in=0] (0,0.6);
    \draw[color=blue,postaction={decorate}] (-1,0) to[out=85,in=180]  (0,0.6);
    \node[circle,fill=black,inner sep=1pt] at (-1,0)(a){} node[shift=(180:0.2)] at (a) {$a$};
    \node[circle,fill=black,inner sep=1pt] at (1,0)(b){} node[shift=(0:0.2)] at (b) {$b$};
    \node[circle,fill=black,inner sep=1pt] at (0,0.6)(c){} node[shift=(90:0.2)] at (c) {$c$};
    \end{scope}
\end{tikzpicture}
\]
\vspace{5pt}
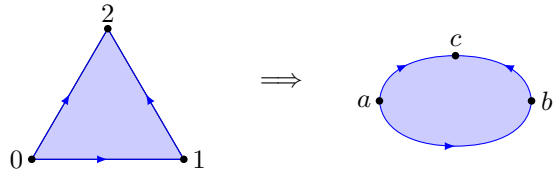
\begin{center}{Figure 3 :\ \ \ \ $\phi_{e_{abc}}(\Delta^{2})$ is obtained from  $\Delta^2$ without any degeneracy on the faces.}\end{center}

\[
\begin{tikzpicture}[scale=1,baseline=(current bounding box.center)]
\begin{scope}[decoration={
    markings,
    mark=at position 0.5 with {\arrow{latex}}}
    ]
    \draw[fill=blue!20, line join=round] (-150:2)--(0,0)--(150:2) -- cycle;
    \draw[blue,postaction={decorate}] (-150:2)--(0,0);
    \draw[blue,postaction={decorate}] (0,0)--(150:2);
    \draw[blue,postaction={decorate}] (-150:2)--(150:2);
     \node[circle,fill=black,inner sep=1pt] at (0,0)(b){} node at +(0:0.2){$1$};
     \node[circle,fill=black,inner sep=1pt] at (-150:2)(a1){} node[shift=(180:0.2)] at (a1) {$0$};
     \node[circle,fill=black,inner sep=1pt] at (150:2)(a2){} node[shift=(180:0.2)] at (a2) {$2$};
\end{scope}
\end{tikzpicture}
\hspace{0.5cm}
\Longrightarrow
\hspace{0.5cm}
\begin{tikzpicture}[scale=1,baseline=(current bounding box.center)]
\begin{scope}[decoration={
    markings,
    mark=at position 0.5 with {\arrow{latex}}}
    ]
    \fill[blue!20] (0,0) to[out=-85,in=-95] (0:2) to[out=95,in=85] (0,0) -- cycle;
    \draw[color=blue,postaction={decorate}] (0,0) to[out=-85,in=-95] (0:2);
    \draw[color=blue,postaction={decorate}] (0:2) to[out=95,in=85] (0,0);
    \node[circle,fill=black,inner sep=1pt] at (0,0)(a){} node at +(180:0.2){$a$};
    \node[circle,fill=black,inner sep=1pt] at (0:2)(b){} node[shift=(0:0.2)] at (b) {$b$};
\end{scope}
\end{tikzpicture}
\]
\vspace{5pt}
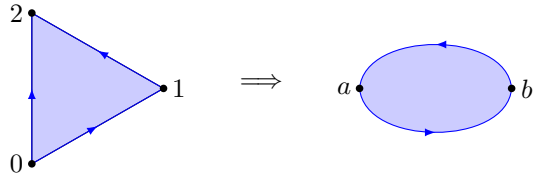
\begin{center}{Figure 4 :\ \ \ \ $\phi_{e_{aba}}(\Delta^{2})$ is obtained from  $\Delta^2$ with degeneracy on the $1$-face ``$02$".}
\end{center}

\[
\begin{tikzpicture}[scale=1.2,baseline=(current bounding box.center)]
\begin{scope}[decoration={
    markings,
    mark=at position 0.5 with {\arrow{latex}}}
    ]
    \fill[fill opacity=0.8, left color=black!20, right color=black!0] (0,0)--(0.2,1)--(2.2,0) -- cycle;
    \fill[fill opacity=0.8,left color=blue!50, right color=blue!10] (2.2,0)--(0.2,1)--(1.2,2) -- cycle;
    \fill[fill opacity=0.8, left color=black!40, right color=black!10] (0,0)--(0.2,1)--(1.2,2) -- cycle;
    \draw[color=black!60,postaction={decorate}] (0,0)--(2.2,0);
    \draw[color=blue!90,postaction={decorate}, dashed] (2.2,0)--(0.2,1);
    \draw[color=blue!40!black!70,postaction={decorate}] (0.2,1)--(1.2,2);
    \draw[postaction={decorate}] (0,0)--(0.2,1);
    \draw[black,postaction={decorate}] (0,0)--(1.2,2);
    \draw[color=blue,postaction={decorate}] (2.2,0)--(1.2,2);
    \node[circle,fill=black,inner sep=1pt] at (0,0)(a1){} node[shift=(-180:0.2)] at (a1) {$0$};
    \node[circle,fill=black,inner sep=1pt] at (2.2,0)(b1){} node[shift=(0:0.2)] at (b1) {$1$};
    \node[circle,fill=black,inner sep=1pt] at (0.2,1)(a2){} node[shift=(180:0.2)] at (a2) {$2$};
    \node[circle,fill=black,inner sep=1pt] at (1.2,2)(b2){} node[shift=(90:0.2)] at (b2) {$3$};
\end{scope}
\end{tikzpicture}
\Longrightarrow
\hspace{0.2cm}
\begin{tikzpicture}[scale=1.2,baseline={([yshift=0pt]current bounding box.base)},decoration={
    markings,
    mark=at position 0.5 with {\arrow{latex}}}
    ]
  \fill[fill opacity=0.8,left color=blue!50, right color=blue!10, line join=round] (0,0) to[out=85,in=95] (2.2,0) to [out=-170,in=-10] (0,0) -- cycle;
  \fill[fill opacity=0.8,left color=black!20, right color=black!0, line join=round] (0,0) to[out=-85,in=-95] (2.2,0) to [out=-170,in=-10] (0,0) -- cycle;
  \draw[black!60,postaction={decorate}] (0,0) to[out=-10,in=-170] (2.2,0);
  \draw[color=black!80] (0,0) to[out=-85,in=-95] (2.2,0);
  \draw[color=blue!80] (2.2,0) to [out=95,in=85] (0,0);
  \draw[color=blue!90,postaction={decorate}, dashed] (2.2,0) to [out=170,in=10] (0,0);
  \node[circle,fill=black,inner sep=1pt] at (0,0)(a){} node[shift=(-180:0.2)] at (a) {$a$};
  \node[circle,fill=black,inner sep=1pt] at (2.2,0)(b){} node[shift=(0:0.2)] at (b) {$b$};
\end{tikzpicture}
\]
\vspace{5pt}
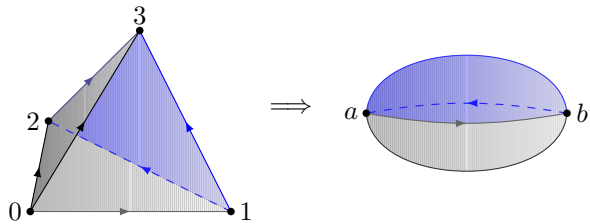
\begin{center}{Figure 5:\ \ \ \  $\phi_{e_{abab}}$ is obtained from  $\Delta^3$ with degeneracy on $2$-faces ``013" and ``023".}
\end{center}

To continue, we shall pick a subcomplex from $S(V)$.  The following definition is crucial for our construction.

\begin{defn}\label{admiss}
\em{
Let $P$ be a path complex over the vertices set $V$, and let $e_n\in P_n$ be an elementary $n$-path on $V$, then $e_n$ is called
\emph{admissible} for $P$ if $e_n$ occurs as a summand in the expression of some allowed $n$-path $p_n\in\Omega_n(P)$, that is to say,
if one writes $p_n=\sum_{j=1}^{s}r_{j}e_{\mathbf{i}^{(j)}}$ with $s$
a finite integer and each nonzero $r_{j}\in R$, $e_{\mathbf{i}^{(j)}}\in P_{n}$, then one has $e_n=e_{\mathbf{i}^{(j)}}$ for some multi-index $\mathbf{i}^{(j)}=i_{0}^{(j)}\cdots i_{n}^{(j)}$.
}
\end{defn}

\begin{rem}
\emph{
Though an elementary $n$-path need not to be in an allowed $n$-path in $\Omega_n(P)$, there is a pure geometric way to find out all the admissible elementary $n$-paths at all dimension levels. To see this, consider
the $n$-cells associated with all elementary $n$-paths in $P_n$. By our construction, any two of them intersect at one $(n-1)$-face at most
(otherwise they coincides). After maybe delete some elementary $n$-paths, one can find a maximal group $\{e_j\}^k_{j=1}$
of elementary $n$-paths such that the boundaries of the glued $n$-cells in $S(V)$ associated with all these $e_j$'s, lies exactly
in the collection of all $(n-1)$-cells associated with the elementary $(n-1)$-paths in $P_{n-1}$. It is not hard to see that any elementary
$n$-path $e_j$ in this maximal group is admissible, the reverse direction is also true, but the proof is technical and tedious so we omit it here.
}
\end{rem}

Now we can give the promised s$\Delta$-complex $S(P)$.

\begin{defn}\label{sp}

\emph{
Let $P$ be a path complex over $V$, an s$\Delta$-complex $S(P)$ is defined as a subcomplex of $S(V)$ consisting of  only the
cells associated with the admissible elementary paths. To be explicit,  $S(P):=\coprod\phi_{e_p}(e_p)/ \sim\ \subset\widetilde{S(P)}$  can
be obtained by glueing all the cells $\phi_{e_p}(e_p)$ via the characteristic maps restricted to the faces, where $e_p$ runs over all admissible
paths at all dimension levels.
}
\end{defn}

By definition, one obtains a simplicial structure on s$\Delta$-complex $S(P)$ as follows.

Suppose $\{e_{\mathbf{i}}=e_{i_{0}\cdots i_{n}}\}$  is an admissible elementary $n$-paths of $P$, for $e_{\mathbf{i}}$ by the above construction we can define a map
$$F_{S}: e_{\mathbf{i}}\mapsto\phi_{e_{\mathbf{i}}}$$ where $\phi_{e_{\mathbf{i}}}$ is the characteristic map
$$\phi_{e_{\mathbf{i}}}: \Delta^n\rightarrow S(P)$$ previously defined.  It is obvious from the definition of $\phi_{e_{\mathbf{i}}}$ that one always
has $\partial_{p}\phi_{e_{\mathbf{i}}}(\Delta^p)=\phi_{e_{\mathbf{i}}}\partial_{p}(\Delta^p)$ where the right hand side map means the restriction of $\phi_{e_{\mathbf{i}}}$.
If we denote $F_{\Delta}=\phi_{e_{\mathbf{i}}}F_S$, then from the construction of $S(P)$ one gets the following map
$$F_{\Delta}: e_{\mathbf{i}}\mapsto S(P).$$
Note also that $F_{\Delta}$ can be extended linearly to become a map
from $\Omega_\ast(P; R)\rightarrow C_\ast(S(P); R)$, where $C_\ast(S(P); R)$ is the cellular (or read as simplicial) chain complex of $S(P)$.

We give a simple example to illustrate the above construction.
\begin{eg}
\emph{Consider the path complex $P$ defined as follows:\\
\indent $0$-paths: $0, 1, 2, 3$\\
\indent $1$-paths: $01, 02, 03, 10, 12, 13$\\
\indent $2$-paths: $010, 012$.\\
By definition and the result in Example \ref{eg3.3}, its geometric realization $S(P)$ is a 2-dimensional s$\Delta$-complex
with two sheets and a handle, which is drawn on Figure 6.}
\end{eg}
\[
\begin{tikzpicture}[scale=1,baseline=(current bounding box.center),decoration={
    markings,
    mark=at position 0.5 with {\arrow{latex}}}]
    \fill[blue!20] (0,0)--(0:3) to[out=95,in=0] (1.5,1.2) to[out=-180,in=85] (0,0);
    \fill[blue!50] (0,0) to[out=55,in=125] (0:3)--(0,0) -- cycle;
    \draw[color=blue,postaction={decorate}] (0,0)--(0:3);
    \draw[color=blue,postaction={decorate}] (0:3) to[out=125,in=55] (0,0);
    \draw[color=blue,postaction={decorate}] (0,0) to[out=85,in=-180] (1.5,1.2);
    \draw[color=blue,postaction={decorate}] (3,0) to[out=95,in=0] (1.5,1.2);
    \draw[color=blue,postaction={decorate}] (0,0)-- (1.5,-1.2);
    \draw[color=blue,postaction={decorate}] (3,0)-- (1.5,-1.2);
    \node[circle,fill=black,inner sep=1pt] at (0,0)(0){} node at +(180:0.2){$0$};
    \node[circle,fill=black,inner sep=1pt] at (0:3)(1){} node[shift=(0:0.2)] at (1) {$1$};
    \node[circle,fill=black,inner sep=1pt] at (1.5,1.2)(2){} node[shift=(90:0.2)] at (2) {$2$};
    \node[circle,fill=black,inner sep=1pt] at (1.5,-1.2)(3){} node[shift=(-90:0.2)] at (3) {$3$};
\end{tikzpicture}
\]
\begin{center}{Figure 6:\ \ \ \  The geometric realization $S(P)$ of $P$.}
\end{center}

\subsection{(Co)homological isomorphism}
\indent

\noindent Now we are ready to prove the main result of this section.
\begin{lem}\label{degeneracy}
Let $e_n$ be a regular elementary $n$-path over a finite set $V$,  and let $e_{n-1}$ be an elementary $(n-1)$-path in the boundary
of $e_n$, then either that $e_{n-1}$ is regular or that the reduced path of $e_{n-1}$ is an elementary $(n-2)$-path which lies in the
boundary of some other regular boundary $(n-1)$-path $e^\prime_{n-1}$ of $e_n$.
\end{lem}

\begin{proof}
Assume that $e_n=e_{i_0\cdots i_n}$ and $e_{n-1}=e_{i_0\cdots i_{p-1}\widehat{i}_pi_{p+1}\cdots i_n}$ in the boundary of $e_n$. If $e_{n-1}$ is non-regular one
has $i_{p-1}=i_{p+1}$, it follows that if we denote $e_{n-2}$ as the elementary $(n-2)$-path $e_{i_0\cdots i_{p-1}\widehat{i}_p\widehat{i}_{p+1}i_{p+2}\cdots i_n}$,
then $e_{n-2}$ is the reduced path of  $e_{n-1}$. To see this one just notices that $e_{n-2}=e_{i_0\cdots i_{p-2}\widehat{i}_{p-1}\widehat{i}_{p}i_{p+1}\cdots i_n}$
is already regular, for otherwise one has $i_{p-2}=i_{p+1}=i_{p-1}$, contradicting to that $e_n$ is regular. Now if $i_{p-2}\neq i_{p}$, then let
$e^\prime_{n-1}=e_{i_0\cdots i_{p-2}\widehat{i}_{p-1}i_{p}\cdots i_n}$ and we are done. Otherwise one has $i_{p-2}=i_{p}$, then if the $n-1$-path
$e_{i_0\cdots i_{p-3}\widehat{i}_{p-2}i_{p-1}\cdots i_n}$ is regular, we let it be $e^\prime_{n-1}$ and the proof is finished, or one would have
$i_{p-3}=i_{p-1}$ and one can consider the $(n-1)$-path $e_{i_0\cdots i_{p-4}\widehat{i}_{p-3}i_{p-2}\cdots i_n}$ to continue this procedure by shifting the
omission index forward until one finds a regular elementary $(n-1)$-path $e^\prime_{n-1}$, this process can certainly be terminated since $e_{\widehat{i_0}\cdots i_{p-1}i_pi_{p+1}\cdots i_n}$
is already regular by Definition \ref{pathcomplex}. The result hence follows.
\end{proof}

\begin{lem} \label{lem3.7}
Given a regular path complex $P$ over a finite set $V$. Let $e_{n}\in P_{n}$, if $e_{n}$ is admissible, then $F_{\Delta}(\partial_{n}(e_{n}))=\partial_{n}(F_{\Delta}(e_{n}))$.
\end{lem}

\begin{proof}
Let $V=\{i_{l}\}$, write $e_{n}=e_{i_{0}i_{1}\cdots i_{n}}$ and denote $e_{\widehat{i_q}}=e_{i_{0}\cdots\widehat{i_{q}}\cdots i_{n}}$. Hence $e_{n}$
gives rise to a distinguished characteristic map $\phi_{e_{n}}: \Delta^n\rightarrow S(P)$. Suppose
$\Delta^{n-1}_{e_{\widehat{i_q}}}$ is the $(n-1)$-face of $\Delta^n$ decided by $e_{\widehat{i_q}}$, and denote $\phi_{e_{\widehat{i_{q}}}}$
as the map of $\phi_{e_{n}}$ restricted to $\Delta^{n-1}_{e_{\widehat{i_q}}}$.
Let us first compute $F_{\Delta}(\partial_{n}(e_{n}))$, recall that by the definition of boundary maps of regular path complexes one has
\begin{align*}
F_{\Delta}\big(\partial_{n}(e_{n})\big)&=F_{\Delta}\bigg(\sum_{\textit{regular\ }e_{\widehat{i_{q}}}}(-1)^{q}e_{\widehat{i_{q}}}\bigg)\\
&=\sum_{\textit{regular\ }e_{\widehat{i_{q}}}}(-1)^{q}F_{\Delta}\big(e_{\widehat{i_{q}}}\big)\label{FDel}\\
&=\sum_{\textit{regular\ }e_{\widehat{i_{q}}}}(-1)^{q}\phi_{e_{\widehat{i_{q}}}}(\Delta^{n-1}_{e_{\widehat{i_q}}}).
\end{align*}
On the other hand one has
\begin{align*}
\partial_{n}\big(F_{\Delta}(e_{n})\big)&=\partial_{n}\big(\phi_{e_{n}}(\Delta^{n})\big)\\
&=\phi_{e_{n}}\big(\partial_{n}(\Delta^{n})\big)\\
&=\phi_{e_{n}}\bigg(\sum_{\textit{regular\ }e_{\widehat{i_{q}}}}(-1)^{q}\Delta^{n-1}_{e_{\widehat{i_q}}}+\sum_{\textit{non-regular\ }e_{\widehat{i_{p}}}}(-1)^{p}\Delta^{n-1}_{e_{\widehat{i_p}}}\bigg)\\
&=\sum_{\textit{regular\ }e_{\widehat{i_{q}}}}(-1)^{q}\phi_{e_{\widehat{i_{q}}}}\big(\Delta^{n-1}_{e_{\widehat{i_q}}}\big)+\sum_{\textit{non-regular\ }e_{\widehat{i_{p}}}}(-1)^{p}\phi_{e_{\widehat{i_{p}}}}\big(\Delta^{n-1}_{e_{\widehat{i_p}}}\big)\\
&=\sum_{\textit{regular\ }e_{\widehat{i_{q}}}}(-1)^{q}\phi_{e_{\widehat{i_{q}}}}\big(\Delta^{n-1}_{e_{\widehat{i_q}}}\big).
\end{align*}
The last equality follows from that if $e_{\widehat{i_p}}$ is non-regular, then
in fact $\phi_{e_{\widehat{i_{p}}}}(\Delta^{n-1}_{e_{\widehat{i_p}}})$ coincidences with an $(n-2)$-face of some other
$(n-1)$-face $\phi_{e_{\widehat{i_{q}}}}\big(\Delta^{n-1}_{e_{\widehat{i_q}}}\big)$ for some regular $e_{\widehat{i_{q}}}$ by Lemma \ref{degeneracy},
it vanishes modulo the gluing relations. The proof is finished by comparing the above two equalities.
\end{proof}

\begin{lem}\label{lem3.8}
The map $F_{\Delta}: \Omega_\ast(P; R)\rightarrow C_\ast(S(P); R)$ is injective.
\end{lem}

\begin{proof}
To show the result, let $p_{n}=e_{i_{p}^{0}\cdots i_{p}^{n}}$ and $q_{n}=e_{i_{q}^{0}\cdots i_{q}^{n}}$ be two admissible elementary $n$-paths such that
$F_{\Delta}(p_{n})=F_{\Delta}(q_{n})$, one just needs to prove that $p_{n}=q_{n}$ by an inductive discussion. We claim that one has $i^{j}_{p}=i^{j}_{q}$ for each $j$. To
show this, suppose $l$ is the least number such that $i^{l}_{p}\neq i^{l}_{q}$. Since $F_{\Delta}(p_{n})=F_{\Delta}(q_{n})$ implies particularly that
they have same vertices, there must be an integer $s$ such that it is the least number satisfying $i^{s}_{p}= i^{l}_{q}$. Now if $s>l$, then the
$(n-l)$-face  $F_{\Delta}\big(e_{i_{q}^{l}\cdots i_{q}^{n}})$  of $F_{\Delta}(q_{n}\big)$ would not appear in $F_{\Delta}(p_{n})$ since the
$(n-s)$-face $F_{\Delta}\big(e_{i_{p}^{s}\cdots i_{p}^{n}}\big)$ is the largest face of $F_{\Delta}(p_{n})$ with the the ordering vertices starting
from $i^{l}_{q}$, a contradiction. So suppose $s<l$, a similar discussion leads to a contradiction, too. That is to say $i^{j}_{p}=i^{j}_{q}$ for each $j$ and it completes the proof.
\end{proof}

Now we can turn to the main result of this section. Hereafter, both of the path (co)homology and simplicial (co)homology
should be read as (co)homologies with coefficients in $R$ if not specified.

\begin{cor}\label{cor3.9}
For a path complex $P$, the map $F_{\Delta}: \Omega_\ast(P; R)\rightarrow C_\ast(S(P); R)$ induces a 1-1-correspondence of cycles between H$_{\ast}(P)$ and H$_{\ast}(S(P))$.
\end{cor}

\begin{proof}
Suppose $c_{n}=F_{\Delta}(p_{n})$ with $p_{n}=\sum_{i=1}^{m}r_{i}e_{n}^i$ ($r_{i}\in R$) a cycle in $\Omega_n(P; R)$ such that each $e_{n}^i$
is an admissible elementary $n$-path. It follows from $\partial_{n}(p_{n})=0$ that one has $F_{\Delta}\big(\partial_{n}(p_{n})\big)=0$.
Thus by Lemma \ref{lem3.7} one has $\partial_{n}\big(F_{\Delta}(p_{n})\big)=0$. Lemma \ref{lem3.8}
then implies $F_{\Delta}(p_{n})=c_{n}=0$,  that is to say, $c_{n}$ is a cycle in $C_n(S(P); R)$.

Conversely, assume that $c_{n}=F_{\Delta}(p_{n})$ is a cycle in $C_n(S(P); R)$, then it follows from $\partial_{n}\big(F_{\Delta}(p_{n})\big)=0$
and Lemma \ref{lem3.7} that $F_{\Delta}\big(\partial_{n}(p_{n})\big)=0$, hence $\partial_{n}(p_{n})=0$ by Lemma \ref{lem3.8}, i.e., $p_{n}$ is
a cycle in $\Omega_n(P; R)$.Thus we obtain the desired 1-1-correspondence.
\end{proof}

\begin{cor}\label{cor3.10}
For a path complex $P$, the map $F_{\Delta}: \Omega_\ast(P; R)\rightarrow C_\ast(S(P); R)$ induces a $1$-$1$-correspondence of boundaries between H$_{\ast}(P)$ and H$_{\ast}(S(P))$.
\end{cor}

\begin{proof}
The proof is very similar to the above one. Suppose $p_{n}=\partial_{n+1}(p_{n+1})=\partial_{n+1}\big(\sum_{i=1}^{m}r_{i}e_{n+1}^i$\big) ($r_{i}\in R$)
is a boundary in $\Omega_n(P; R)$ such that each $e_{n+1}^i$ is an admissible elementary $(n+1)$-path. Denote $c_n=F_{\Delta}(p_{n})$,
then by Lemma \ref{lem3.7} one gets $c_n=F_{\Delta}\big(\partial_{n+1}(p_{n+1})\big)=\partial_{n+1}\big(F_{\Delta}(p_{n+1})\big)$. This shows that
$c_n$ is a boundary in $C_n(S(P); R)$.

Conversely, assume that $c_{n}=\partial_{n+1}\big(F_{\Delta}(p_{n+1})\big)$ is a boundary in $C_n(S(P); R)$ for some $p_{n+1}\in\Omega_{n+1}(P; R)$,
then $p_{n}=\partial_{n+1}(p_{n+1})$ is a boundary in $\Omega_n(P; R)$ and apparently one has $c_{n}=F_{\Delta}(p_{n})$ by Lemma \ref{lem3.7}, which completes the proof.
\end{proof}

We conclude this section by the two main theorems as follows.
\begin{mytheo}\label{thm3.10}
Let $P$ be a path complex and $R$ be a ring, the map $F_\Delta$ induces an isomorphism $H_\ast(P)\cong H_\ast(S(P))$.
\end{mytheo}

\begin{proof}
Follows directly from Corollaries \ref{cor3.9} and \ref{cor3.10}.
\end{proof}

Theorem \ref{thm3.10} together with Mayer-Vietoris exact sequence can be used to simplify many proofs of results in Section 5 of \cite{Grig1}.
Moreover, as a direct application, now let us sketch a new proof the main theorem of \cite{Grig2}. As we see in Example \ref{Exam},
one can associate each digraph $G$ with a path complex $P(G)$, hence it is reasonable to say the path (co)homology of a digraph.
Now for any finite simplicial complex $S$, \cite{Grig2} gives a natural way to construct a finite \emph{cubical digraph} $G_S=(V, E)$.
In details, the set $V$ of vertices of $G_{S}$ coincides with the set of all simplices from $S$, and two simplices $s$, $t$ are connected
in $G_{S}$ by a directed edge $(s\rightarrow t)\in E$ if and only if $s\supset t$ and dim$(s)=$dim$(t)+1$.
\[
\begin{tikzpicture}[scale=1.5,baseline=(current bounding box.center)]
    \draw[color=blue,fill=blue!20, line join=round] (-150:1.2)--(-30:1.2)--(90:1.2) -- cycle;
    \draw[color=blue] (-150:1.2)--+(150:1.2);
     \node[circle,fill=black,inner sep=1pt] at (90:1.2){};
     \node[circle,fill=black,inner sep=1pt] at (-150:1.2){};
     \node[circle,fill=black,inner sep=1pt] at (-30:1.2){};
     \node[circle,fill=black,inner sep=1pt,shift=(150:1.8)] at (-150:1.2){};
     \node at (-90:1) {$G$};
\end{tikzpicture}
\hspace{3cm}
\begin{tikzpicture}[scale=1.5,>=latex,baseline=(current bounding box.center)]
     \fill[blue!20] (-150:1.2)--(-30:1.2)--(90:1.2) -- cycle;
     \node[circle,fill=black,inner sep=1.5pt] at (0,0)(s){} node at +(-45:0.2){$s$};
     \node[circle,fill=black,inner sep=1.25pt] at (150:0.6)(t){} node[shift=(180:0.2)] at (t){$t$};
     \node[circle,fill=black,inner sep=1.25pt] at (30:0.6)(a){} node[shift=(90:0.2)] at (a) {};
     \node[circle,fill=black,inner sep=1pt] at (90:1.2)(2){};
     \node[circle,fill=black,inner sep=1pt] at (-150:1.2)(0){};
     \node[circle,fill=black,inner sep=1pt,shift=(150:0.9)] at (-150:1.2) (u){};
     \node[shift=(90:0.2)] at (u) {$u$};
     \node[circle,fill=black,inner sep=1.25pt] at (-90:0.6)(b){};
     \node[circle,fill=black,inner sep=1pt] at (-30:1.2)(1){};
     \node[circle,fill=black,inner sep=1pt,shift=(150:1.8)] at (-150:1.2) (v){};
     \node[shift=(90:0.2)] at (v) {$v$};
     \draw[red,->] (u)--(0);
     \draw[red,->] (s)--(a);
     \draw[red,->] (s)--(t);
     \draw[red,->] (s)--(b);
     \draw[red,->] (u)--(v);
     \draw[red,->] (u)--(0);
     \draw[red,->] (b)--(0);
     \draw[red,->] (b)--(1);
     \draw[red,->] (a)--(1);
     \draw[red,->] (a)--(2);
     \draw[red,->] (t)--(0);
     \draw[red,->] (t)--(2);
     \node at (-90:1) {$G_S$};
\end{tikzpicture}
\]
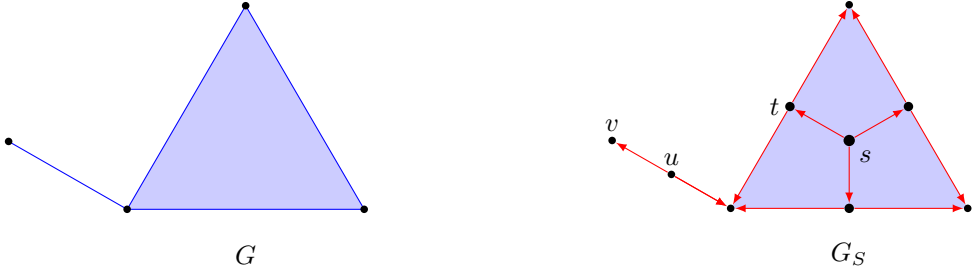
\begin{center}{Figure 7:\ \ \ \ A simplicial complex $S$ and its cubical digraph $G_S$.}
\end{center}
One can verify that (but we omit the details here) the geometric realization of $P(G_{S})$ is exactly the full barycentric subdivision $B_S$ of
$S$, which obvious has the same simplicial homology groups as those of $S$. Thus Theorem \ref{thm3.10} implies the isomorphism in \cite[Theorem 5.1]{Grig2}.

As the dual version of Theorem \ref{thm3.10}, we have

\begin{mytheo}\label{thm3.11}
Let $P$ be a path complex and  $R$ be a unital ring. The map $F_{\Delta}$ induces an isomorphism between the path cohomology of $P$ and the
simplicial cohomology of $S(P)$, namely, we have an isomorphism H$^{\ast}(P)\cong$H$^{\ast}(S(P))$.
\end{mytheo}

\begin{proof}
Recall that by definition, $H^{\ast}(P)$ is defined as the homology group of the following cochain complex:
\[
\Omega^\ast(P; R):=\cdots\leftarrow\Omega^n(P; R)\leftarrow\Omega^{n-1}(P; R)\leftarrow\cdots\leftarrow\Omega^0(P; R)\leftarrow0 \tag{3.1}\label{coch}.
\]
Using the isomorphism $\Omega^p(P; R)\cong\mbox{Hom}_R(\Omega_p(P; R), R)$
in (\ref{fome}) and the obvious isomorphism $\Omega_p(P; R) \cong R\otimes_\mathbb{Z}\Omega_p(P; \mathbb{Z})$,
by adjoint isomorphism $\mbox{Hom}_R(R\otimes_\mathbb{Z}\Omega_p(P; \mathbb{Z}), R)\cong\mbox{Hom}_\mathbb{Z}(\Omega_p(P; \mathbb{Z}),
R)$ one has $\Omega^p(P; R)\cong\mbox{Hom}_\mathbb{Z}(\Omega_p(P; \mathbb{Z}); R)$.
That is to say, (\ref{fome}) and (\ref{coch}) imply a cochain isomorphism
\[
\Omega^\ast(P; R)\cong\mbox{Hom}_\mathbb{Z}(\Omega_\ast(P; \mathbb{Z}); R).\tag{3.2}\label{cois}
\]

On the other hand, by definition $H^{\ast}(S(P))$ is the homology groups of the cochain complex
\[
C^\ast(S(P); R):=\mbox{Hom}_\mathbb{Z}(C_\ast(S(P); \mathbb{Z}), R)\tag{3.3}\label{coho}
\]
where $C_\ast(S(P); \mathbb{Z})$ is the simplicial chain complex of $S(P)$. Now Theorem \ref{thm3.10} asserts an isomorphism between the
homology groups of two chain complexes of free abelian groups:
\[
F_\Delta: \Omega_\ast(P; \mathbb{Z})\rightarrow C_\ast(S(P); \mathbb{Z}).\tag{3.4}\label{diso}
\]
Thus (\ref{cois}), (\ref{coho}) and (\ref{diso}) imply the desired isomorphism by \cite[Corollary 3.4]{Hatc}.
\end{proof}

At last let us make a comment on Theorems \ref{thm3.10} and \ref{thm3.11}. As one may easily see that, for any complete digraph $G$ arising from an
ordered $n$-simplex $\Delta^n$, $\Delta^n$ is the geometric realization of $P(G)$. One should also note that in usual, though
$F_\Delta$ induces an isomorphism between the (co)homologies of them, the realization map $F_\Delta$ itself does not induce a (co)chain
isomorphism between $\Omega_\ast(P(G))$ (resp. $\Omega^\ast(P(G))$) and $C_\ast(S(P(G)); R)$ (resp. $C^\ast(S(P(G)); R)$) except for $G$ with
some strict conditions (for e.g. $G$ is complete, or of the case occurred in \cite[Theorem 5.24]{Grig3}).

\section{Application \textrm{I}: Relationship with Hochschild (co)homology}
\noindent Starting from here, we will give some applications of the main results obtained in Section 3. The aim of this section is to connect
path (co)homology with Hochschild (co)homology. This would allow to give a combinatoric description
of Hochschild (co)homology theory, while the latter owns a pure algebraic definition involving tensor product.

Assume in the sequel that $R$ is a commutative unital ring. We shall associate a path complex $P$ with two associative unital algebras
$A_{S(P)}$ and $\bar{A}_{S(P)}$ and consider the corresponding path (co)homology and Hochschild (co)homology. We recall from \cite{Gers} that,
for any (locally) finite simplicial complex $S$, one can define a digraph $E_{S}$ where the vertices are all simplices from $S$ and a couple
$(s, t)$ is a directed edge if and only if $s\supset t$ (compare with the definition of cubical digraph stated above Figure 5). The $R$-algebra
$A_{S}$ is defined as a set of all finite $R$-linear combinations of edges of $E_{S}$ with a multiplication given by the rule:
\[
(s_1, t_1)(s_2, t_2)=\left\{
\begin{array}{ll}
(s_1, t_2)  &\mbox{if $t_1=s_2$;}\\
0           &\mbox{otherwise.}
\end{array}
\right.
\]
Moreover, if we replace the elements in $A_{S}$, i.e., all the edges of
$E_{S}$ by arbitrary couples $(s, t)$ whenever $s,t$ are vertices in $E_{S}$, and keep
the multiplication formula as above, then we obtain a new $R$-algebra different from $A_{S}$,
which is denoted by $\bar{A}_{S}$. It is obvious that both of $A_{S}$
and $\bar{A}_{S}$ are free as $R$-modules.

The notion of Hochschild (co)homology groups was first defined in \cite{Hoch} for the algebra over a field, and was later extended in
\cite{Cart} for the algebra over a commutative unital ring. Now we recall some of the definitions. Let $A$ be a $R$-algebra which is also a
projective $R$-module and let $M$ be an $A$-bimodule. Denote $\tilde{S}_n(A)$ the $R$-module obtained by taking the $n$-fold tensor product of
$A$ over $R$, which is easily checked to be $R$-projective. The Hochschild homology groups $HH_\ast(A, M)$ of $A$ with coefficients in $M$ are
referred to as the homology groups of the complex $C_\ast(A, M)=M\otimes_K\tilde{S}_\ast(A)$ with differentiation
\begin{align*}
d_n(m\otimes\lambda_{1}\otimes\cdots\otimes\lambda_{n})&= m\lambda_{1}\otimes\lambda_{2}\otimes\cdots\otimes\lambda_{n}+\sum^{n-1}_{i=1}(-1)^{i} m\otimes\lambda_{1}\otimes\cdots\otimes \lambda_i\lambda_{i+1}\otimes\cdots\otimes\lambda_n\\
&\quad+(-1)^{n}\lambda_n m\otimes\lambda_{1}\otimes\cdots\otimes\lambda_{n-1}.
\end{align*}

Similarly, set $C^n(A, M)$ as the set of all $R$-linear functions $f:\tilde{S}_n(A)\rightarrow M$. The Hochschild cohomology groups $HH^\ast(A,
M)$ of $A$ with coefficients in $M$ are referred to as the cohomological groups of the cochain complex $C^\ast(A, M)$ with the coboundary
operator $\delta: C^n(A, M)\rightarrow C^{n+1}(A, M)$ given by the following formula:
\begin{align*}
\delta f(x_{1}, x_{2},\cdots,x_{n+1})&= x_{1}f (x_{2},\cdots,x_{n})+
\sum^n_{i=1}(-1)^{i} f(x_{1},\cdots,x_{i}x_{i+1},\cdots,x_{n})\\
&\quad+(-1)^{n+1} f(x_{1},\cdots,x_{n-1})x_{n}.
\end{align*}
In particular, if $M=A$, then we shall write them shortly by $HH_\ast(A)$ and $HH^\ast(A)$ as no confusion rises.

We state the first result as the following theorem.

\begin{mytheo}\label{thm4.2}
Let $P$ be a finite path complex and $R$ be a commutative unital ring, there exists an isomorphism H$^\ast(P)\cong$HH$^\ast(A_{S(P)})$.
\end{mytheo}

\begin{proof}
First note that $S(P)$ admits a finite simplicial structure since $P$ is finite. Recall the main theorem in \cite{Gers} which implies that
for the functor defined via $S\mapsto A_S$ from locally finite simplicial complexes to associative unital $R$-algebras, there is an
induced isomorphism $H^\ast(S; R)\cong HH^\ast(A_S)$.  Then, following from this theorem and Theorem \ref{thm3.11} directly, we obtain the claimed isomorphism for $S=S(P)$.
\end{proof}

To derive a similar result for path homology and Hochschild homology, we
shall use the following result, which can be viewed as the dual version
of the main theorem of \cite{Gers}. In the proof we shall use
some results and notations in \cite{Grig3}. In fact, the results there
reveal that the path (co)homology theory is also a powerful tool when
dealing with the algebraic aspect of the simplicial (co)homology, as
one obviously sees that it allows a direct proof of the main theorem of \cite{Gers} by avoid
using the diagram cohomology and the Cohomology Comparison Theorem (see
\cite{Gers} for details).

Given a finite simplicial complex $S$, consider the associated cubical digraph $G=G_S=(V, E)$ (see Figure 4 and the definition above it)
and the path complex $P(G)$ which gives naturally a chain (cochain) complex $\Omega_\ast(G)=\Omega_\ast(P(G))$
($\Omega^\ast(G)=\Omega^\ast(P(G))$) by definition. Also consider the product
digraph $\widetilde{G}=(\widetilde{V}, \widetilde{E})$ where $\widetilde{V}=V\times V$
and $\widetilde{E}$ denotes the set of the edges $(i, j)\rightarrow(i^\prime, j^\prime)$ whenever $j\rightarrow i^\prime$ is in $E$. One can
also associate $P(\widetilde{G})$ with a chain complex $\widetilde{\Omega}_\ast(G)=\Omega_\ast(P(\widetilde{G}), R)$ and a cochain complex
$\widetilde{\Omega}^\ast(G)=\Omega^\ast(P(\widetilde{G}), R)$, respectively. With all these in hand, we can now prove directly the mentioned
dual result, also without involving diagram cohomology and Cohomology Comparison Theorem .

\begin{lem}\label{lem4.3}
Let $S$ be a finite simplicial complex, then there exists an isomorphism  $$H_\ast(S)\cong HH_{\ast}(A_{S}).$$
\end{lem}

\begin{proof}
Let us recall the cochain map
\[
C^\ast(A_S, A_S)\hookrightarrow C^\ast(A_S, \bar{A}_S)\stackrel{\phi}{\cong}\widetilde{\Omega}^\ast(G)\stackrel{\varphi}{\rightarrow} C^\ast(B_S; R)\tag{4.1}\label{iso2}
\]
constructed in \cite{Grig3} where $\varphi$ is a composition map of a sequel of quotient maps and
isomorphism (see \cite[(4)]{Grig3}) and the maps below it), and note that all the maps in
(\ref{iso2}) are also homomorphisms between cochain complexes of $R$-modules and preserve the cohomologies by the results of \cite{Grig3}. Also
note that the isomorphism $\phi$ identifies $C^\ast(A_S, \bar{A}_S)$ with $\widetilde{\Omega}^\ast(G)$, whence $\phi$ identifies $C^\ast(A_S, A_S)$
with a subcomplex of $\widetilde{\Omega}^\ast(G)$ (see the proofs of \cite[Lemmas 5.1 and 5.2]{Grig3} for details), and it is obvious that both
of $C^\ast(A_S, A_S)$ and $C^\ast(A_S, \bar{A}_S)$ are free $R$-modules under
this identification, hence all of $R$-modules in (\ref{iso2}) are free. Thus
applying the functor Hom$_R(-, R_R)$ to sided-terms in (\ref{iso2}) and by Universal Coefficient Theorem for cohomology (see, for example
\cite[3.6.5]{Weib}), one has the following isomorphism:
\begin{align*}
H_n\big(\mbox{Hom}_R(C^\ast(A_S, A_S), R)\big)
\cong H_n\big(\mbox{Hom}_R(C^\ast(B_S, R), R)\big), \tag{4.2}\label{iso3}
\end{align*}
note that ``cohomology" becomes ``homology" here since $C^\ast(A_S, A_S)$
and $C^\ast(B_S, R)$ themselves are cochain complexes. On the one hand
one computes that

\begin{align*}
H_n\big(\mbox{Hom}_R(C^\ast(A_S, A_S), R)\big)&=H_n\big(\mbox{Hom}_R(\mbox{Hom}_R(\tilde{S}_\ast(A_S), A_S), R)\big)\\
&\cong H_n\big(\mbox{Hom}_R(A_S, R_R)\otimes_R\tilde{S}_\ast(A_S)\big)\tag{4.3}\label{iso4}\\
&\cong H_n\big(A_S\otimes_R\tilde{S}_\ast(A_S)\big)\\
&=HH_n(A_S, A_S)
\end{align*}
where the first isomorphism follows from \cite[Proposition 5.2]{Cart}
since $A_S$ is finitely generated free giving that $\tilde{S}_\ast(A_S)$
is also finitely generated free, and the second isomorphism follows from
the obvious isomorphism $\mbox{Hom}_R(A_S, R_R)\cong A_S$ of $A_S$-bimodules.
On the other hand using \cite[Proposition 5.2]{Cart} again one computes that
\begin{align*}
H_n\big(\mbox{Hom}_R(C^\ast(B_S, R), R)\big)&=H_n\big(\mbox{Hom}_R(\mbox{Hom}_\mathbb{Z}(C_\ast(B_S, \mathbb{Z}), R), R)\big)\\
&\cong H_n\big(\mbox{Hom}_R(R, R)\otimes_\mathbb{Z}C_\ast(B_S, \mathbb{Z})\big)\\\tag{4.4}\label{iso5}
&\cong H_n\big(R\otimes_\mathbb{Z}C_\ast(B_S, \mathbb{Z})\big)\\
&=H_n(B_S; R).
\end{align*}
Hence the result follows from (\ref{iso3}), (\ref{iso4}), (\ref{iso5}) and the obvious isomorphism $H_\ast(B_S; R)\cong H_\ast(S; R)$.
\end{proof}

Now we conclude this section by the dual result of Theorem \ref{thm4.2}.
\begin{mytheo}
Let $P$ be a finite path complex and $R$ be a commutative unital ring, there exists an isomorphism H$_\ast(P)\cong$HH$_\ast(A_{S(P)})$.
\end{mytheo}

\begin{proof}
Note that $S(P)$ admits a finite simplicial structure, then the result follows immediately by Theorem
\ref{thm3.10} and Lemma \ref{lem4.3}.
\end{proof}

\section{Application \textrm{II}: K\"{u}nneth formula of path homology}
\noindent Now we turn to the functorial properties of path homology. Throughout this section, let $R$ be a commutative ring.

Similar to simplicial homological theory, the authors of \cite{Grig1} defined the Cartesian product for two path complexes (see \cite[Definition 7.3]{Grig1}),
and furthermore gave the analogues of the Eilenberg-Zilber theorem and K\"{u}nneth formula  for regular path homology with coefficients in a field $K$.

In this section, we will give some geometric interpretation of these results in a more general setting, i.e., for path homology with
coefficients in a commutative ring. Their proofs are based on the correspondence of the path complexes and s$\Delta$-complexes presented
in Section 3. With the same assumptions, we show that the pattern of these proofs can be used to give the K\"{u}nneth formula for the
join of two regular path complexes (see Definition \ref{def5.10} below).

Note that such two types of K\"{u}nneth formula are obtained in \cite{Grig1} in different ways for regular path complexes {\bf over a field} $K$. Here we shall show that
they can be obtained in a unified way for regular path complexes even \textbf{over a commutative ring} $R$.

\subsection{  The case of Cartesian product}

\indent

\noindent For our purpose, we shall first recall the construction of the simplicial cross product in terms of s$\Delta$-complex.
Given two standard simplices $\Delta^{m}$ and $\Delta^{n}$, let us first subdivide $\Delta^{m}\times\Delta^{n}$
into simplices.

We label the vertices of $\Delta^{m}$ as $v_{0}$, $v_{1}$, $\cdots$, $v_{m}$ and the vertices of $\Delta^{n}$ as
$w_{0}$, $w_{1}$, $\cdots$, $w_{n}$. Naturally, we label the $mn$ vertices of $\Delta^{m}\times\Delta^{n}$ as $(v_{0},
w_{0}),(v_{0}, w_{1}), (v_{1}, w_{0}),\cdots, (v_{m}, w_{n})$. We now view the pairs $(i, j)$ with $0\leq i\leq m$ and $0\leq j\leq n$ as the
vertices of an $m\times n$ rectangle grid in $\mathbb{R}^{2}$.

For any path $\sigma$ formed by a sequence of $m+n$ horizontal and vertical
edges in the grid starting at the origin (0,0) and ending at $(m, n)$, we associate it with a standard $(m+n)$-simplex $\Delta^{m+n}_{\sigma}$
in $\Delta^{m}\times\Delta^{n}$ whose vertices are $(v_{i_{k}}, w_{j_{k}})$ for all the $k^{th}$ vertices $(i_{k}, j_{k})$ of the path $\sigma$,
assigned the order ``$<$" such that $(v_{i_{k}}, w_{j_{k}})<(v_{i_{l}}, w_{j_{l}})$ for $i_{k}\leq i_{l}, j_{k}< j_{l}$ or
$i_{k}<i_{l}, j_{k}\leq j_{l}$.

Define a linear map (which can be viewed as an inclusion) $l_{\sigma}:
\Delta^{m+n}_{\sigma}\rightarrow \Delta^{m}\times\Delta^{n}$ by sending the $k^{th}$ vertex of $\Delta^{m+n}_{\sigma}$ to the vertex $(v_{i_{k}}, w_{j_{k}})$.
When $\sigma$ runs through all the possible $m+n$ step-like paths in the grid from (0,0) to $(m, n)$, all the
associated $(m+n)$-simplices $\Delta^{m+n}_{\sigma}$ fit together by $l_\delta$ nicely to form a $\Delta$-structure on
$\Delta^{m}\times\Delta^{n}$ (see, for e.g. \cite[p.278]{Hatc} or \cite[p.68]{Eile}).

Now we can define the \emph{cross product} of simplices by the formula
\[
\Delta^{m}\otimes \Delta^{n}\stackrel{\times}{\longrightarrow}\sum_{\sigma}(-1)^{|\sigma|}\Delta^{m+n}_{\sigma}\ \tag{5.1}\label{doti}
\]
where $|\sigma|$ is the number of squares in the grid lying below the path
$\sigma$, and `$\times$' here means the cross product. The following
boundary formula can be easily verified by a direct calculation.
\[
\partial(\Delta^{m}\times \Delta^{n})=\partial(\Delta^{m})\times\Delta^{n}+(-1)^{m}\Delta^{m}\times \partial(\Delta^{n}).\tag{5.2}\label{ptim}
\]
In particular, let $X$ and $Y$ be s$\Delta$-complexes, $\{\phi^{m}_{\alpha}\}: \Delta^m\rightarrow X$ and $\{\phi^{n}_{\beta}\}: \Delta^n\rightarrow Y$ be the
characteristic maps of $X$ and $Y$, respectively. we denote
by $C_{m}(X; R)$ and $C_{n}(Y; R)$ the simplicial $n$-chain groups with coefficient $R$
of $X$ and $Y$ generated by all $\phi^{m}_{\alpha}$ and $\phi^{n}_{\beta}$, respectively. Then (\ref{doti})
gives rise to the following definition of simplicial cross product
\[
C_{m}(X; R)\otimes_{R} C_{n}(Y; R)\stackrel{\times}{\longrightarrow}C_{m+n}(X\times Y; R)\tag{5.3} \label{dtimp}
\]
by the formula
$$\phi^{m}_{\alpha}\times \phi^{n}_{\beta}=\sum_{\sigma}(-1)^{|\sigma|}(\phi^{m}_{\alpha}\times \phi^{n}_{\beta})l_{\sigma}$$
where $(\phi^{m}_{\alpha}\times \phi^{n}_{\beta})l_{\sigma}$
means the map $(\phi^{m}_{\alpha}\times\phi^{n}_{\beta})$ restricted on
$l_{\sigma}(\Delta^{m+n}_{\sigma})$. By (\ref{ptim}) we have the boundary operators
\[
\partial(\phi^{m}_{\alpha}\times \phi^{n}_{\beta})=\partial(\phi^{m}_{\alpha})\times\phi^{n}_{\beta}+(-1)^{m}\phi^{m}_{\alpha}\times \partial(\phi^{n}_{\beta}).\tag{5.4}\label{ptimp}
\]

\begin{lem}\label{lem5.1}
The chain map $G$: $C_\ast(X; R)\otimes_{R} C_{\ast}(Y; R))\rightarrow C_{\ast}(X\times Y; R)$ induced by (\ref{dtimp}) is injective for any
two s$\Delta$-complexes $X$ and $Y$.
\end{lem}

\begin{proof}
$G$ is a chain map follows directly by the boundary formula of tensor products of two complexes and (\ref{ptimp}). To show that it is injective,
assume $m+n=k$, and suppose that $\Delta^{m}$ has vertices labeled as $v_{0}\cdots v_{m}$, $\Delta^{n}$ has vertices labeled as $w_{0}\cdots w_{n}$
and the vertices of $\Delta^{m+n}_{\sigma}$ are labeled as the ordering vertices of the step-like path $\sigma=(v_{0},
w_{0})\cdots(v_{i_{l}}, \beta_{j_{l}})\cdots(v_{m}, w_{n})$, and let $f^j_m: \Delta^{m}\rightarrow X$ and $g^j_n: \Delta^{n}\rightarrow Y$ be
respectively the characteristic maps of $X$ and $Y$. Let $w=\sum_{m+n=k}\sum_jc^j_{mn}(f^j_m\otimes g^j_n)$, we need to show that $G(w)=0$
implies all $c^j_{mn}=0$.

Indeed, by definition one has
\begin{align*}
G(w)=\sum_{m+n=k}\bigg(\sum_{\delta}(-1)^{|\delta|}\sum_jc^j_{mn}(f^j_m\times g^j_n)l_{\delta}\bigg). \tag{5.5}\label{cros}
\end{align*}
Note that the interiors of all $(f^j_m\times g^j_n)l_\delta(\Delta^{m+n}_{\sigma})$ are different since all $f^j_m$ and $g^j_n$ are different
characteristic maps of $X$ and $Y$, it follows that the map $G(w)$ in $C_{k}(S(X)\times S(Y);
R)$ is totally decided by the effect it acts on all $\Delta^{m+n}_{\sigma}$. Therefore if (\ref{cros}) vanishes then it implies that
$c^j_{mn}(f^j_m\times g^j_n)l_\delta=0$ for each $\Delta^{m+n}_{\sigma}$, or equivalently $c^j_{mn}=0$ since $f^j_m\times g^j_n$ as a sum of
characteristic maps over $X\times Y$ is never a zero map. This finishes the proof.
\end{proof}

Next let us recall the definition of the \emph{cross product} of
two path complexes from \cite{Grig1} and compare it with the above definition.

Let $X$, $Y$ be two finite sets, and let $Z=X\times Y$.
For two elementary $m$- and $n$-paths $e_{i_{0}\cdots i_{m}}\in \mathcal{R}_{m}(X)$ and
$e_{j_{0}\cdots j_{n}}\in\mathcal{R}_{n}(Y)$, as we have already done to the cross product of simplices, again we have an $m\times n$ rectangle grid
in $\mathbb{R}^{2}$ with the vertices $(i_{0}, j_{0}), (i_{0}, j_{1}), (i_{1}, j_{0}),\cdots, (i_{m}, j_{n})$
with the order as before. Similarly for each $m+n$ step-like edgepath $\sigma$ with vertices $(i_{0}, j_{0}), \cdots, (i_{k}, j_{l}), \cdots, (i_{m}, j_{n})$,
we can associate it with an elementary $(m+n)$-path
$e_{\sigma}=e_{(i_{0}, j_{0})\cdots(i_{k}, j_{l})\cdots(i_{m}, j_{n})}$. Furthermore, we define $$e_{i_{0}\cdots i_{m}}\times e_{j_{0}\cdots j_{n}}=\sum_{\sigma}(-1)^{|\sigma|}e_{\sigma},$$
where $|\sigma|$ has the same meaning as before. This formula can be extended
bilinearly to give the cross product $u\times v$ of any two paths $u\in \mathcal{R}_{m}(X)$
and $v\in \mathcal{R}_{n}(Y)$, and the differential operators are give by the
following formula (see \cite[Proposition 7.2]{Grig1}):
$$\partial(u\times v)=\partial(u)\times v+(-1)^{m}u\times\partial(v).$$
More generally, for any two path complexes $P(X)$, $P(Y)$, recall (see
\cite[Definition 7.3]{Grig1}) that their \emph{Cartesian product} $P(Z)=P(X)\boxplus
P(Y)$ are defined as the path complex over $Z$ with each elementary
$k$-path of $P(Z)$ of the form $e_{\sigma}$, where the $k$
step-like path $\sigma$ comes from the above construction for any elementary
$s$-path $e_{\alpha}=e_{i_{0}i_{1}\cdots i_{s}}\in P_{s}(X)$ and elementary
$(k-s)$-path $e_{\beta}=e_{j_{0}j_{1}\cdots j_{k-s}}\in P_{k-s}(Y)$ with
all $i_{m}\in X$ and $j_{n}\in Y$. Also recall that $\mathcal{A}_{n}(Z)$
is the set of all $R$-linear combinations of elements in $P_n(Z)$ and
$\Omega_{n}(Z)$ is defined as the set $\{z_{n}|z_{n}\in\mathcal{A}_{n}(Z)$
and $\partial_{n}(z_{n})\in\mathcal{A}_{n-1}(Z).\}$.

\begin{lem}\label{lem5.2}
$\Omega_{s}(X)\times\Omega_{k-s}(Y)\subseteq\Omega_{k}(Z)$.
\end{lem}
\begin{proof}
It can be proved easily  from the
boundary formula for cross products.
\end{proof}

The following lemma comes from \cite[Proposition 7.12]{Grig1}, whose proof can be carried over to the case where $R$ is a commutative
ring without any change.

\begin{lem}\label{lem5.3}
Any path $w\in\Omega_{\ast}(Z)$ admits a representation
$$w=\sum_{e_{x}\in P(X),\ e_{y}\in P(Y)}c^{xy}(e_{x}\times e_{y})$$
with finitely nonzero coefficients $c^{xy}\in R$ which are uniquely determined by $w$.
Furthermore, the cross products $\{e_{x}\times e_{y}\}$ across all
$e_{x}\in P(X)$ and $e_{y}\in P(Y)$ are linearly independent.
\end{lem}

Our proof in the sequel depends on the following key lemma (comparing its proof with that of \cite[Theorem 7.15]{Grig1}).

\begin{lem}\label{keylem5}
Any path $w\in\Omega_{n}(Z)$ can be written as a finite sum:
$$w=\sum_k\sum_{i\leq n}\big(p_i^{k}(X)\times q_{n-i}^k(Y)\big)$$
where $k$ runs a finite set, each $p_i^{k}(X)\in\Omega_{i}(X)$ and $q_{n-i}^{k}(Y)\in\Omega_{n-i}(Y)$.
\end{lem}

\begin{proof}
Let $w$ be an $n$-path in $\Omega_{n}(Z)$, by Lemma \ref{lem5.3} we can write it as a finite sum:
\[
w=\sum_{i=s}^n\Bigg(\sum_{e_{x}\in J_{i}(X),\ e_{y}\in J_{n-i}(Y)}c^{xy}(e_{x}\times e_{y})\Bigg) \tag{5.6}\label{wori}
\]
where $0\neq c^{xy}\in R$, $J_{i}(X)\subseteq P_{i}(X)$, $J_{n-i}(Y)\subseteq P_{n-i}(Y)$ and $s\geq 0$ is the lowest index of $e_x$ appeared
in the expression. We do the proof by induction on the total number $a$ of the summands $c^{xy}(e_{x}\times e_{y})$ in (\ref{wori}).

If $a=1$, that is to say, $w=c^{xy}(e_{x}\times e_{y})\in\Omega_n(Z)$
where $e_{x}\in P_{s}(X)$ and $e_{y}\in P_{n-s}(Y)$. $\partial(w)=c^{xy}\big(\partial(e_{x})\times e_{y}+(-1)^se_x\times \partial(e_{y})\big)\in
\mathcal{A}_n(Z)$ and Lemma \ref{lem5.2} imply that $\partial(e_{x})\in\mathcal{A}_{s-1}(X)$ and $\partial(e_{y})\in\mathcal{A}_{n-s-1}(Y)$.
Let $k=s$, then $p_s=c^{xy}e_x$ and $q_{n-s}=e_y$ give the desired result.

Now suppose that for any $a<b$ $(b>1)$ the result holds, we shall show
that the result also holds for $a=b$. Write $w$ as in (\ref{wori}), suppose
the total number of the summands is $b$. We rewrite $w$ as follows:
\begin{align*}
w&=\sum_{i=s}^n\Bigg(\sum_{e_{y}\in P_{n-i}(Y)}\bigg(\sum_{e_{x}\in J^y_{i}(X)}c^{xy}e_{x}\bigg)\times e_{y}\Bigg)\\
&=\sum_{i=s}^n\sum_{J^y_{i}(X),J^x_{n-i}(Y)}\Bigg(\bigg(\sum_{e_{x}\in J^y_{i}(X)}c^{xy}e_{x}\bigg)\times\bigg(\sum_{e_{\tilde{y}}\in J^x_{n-i}(Y)}e_{\tilde{y}}\bigg)\Bigg)\tag{5.7}\label{w2}\\
&=\sum_{i=s}^n\sum_{J^y_{i}(X),J^x_{n-i}(Y)}\Bigg(\sum_{e_{x}\in J^y_{i}(X)}\bigg(c^{xy}e_{x}\times\sum_{e_{\tilde{y}}\in J^x_{n-i}(Y)}e_{\tilde{y}}\bigg)\Bigg)
\end{align*}
where each set $J^y_{i}(X)\subseteq P_{i}(X)$ is decided by $e_y$, and
each set $J^x_{n-i}(Y)\subseteq P_{n-i}(Y)$ contains $e_y$ and is decided by $J^y_{i}(X)$.
Note that $\bigcap J^y_{i}(X)=\bigcap J^x_{n-i}(Y)=\emptyset$, $\bigcup
J^y_{i}(X)=J_{i}(X)$ and $\bigcup J^x_{n-i}(Y)=J_{n-i}(Y)$. It follows from
$w\in\Omega_{n}(Z)$ and (\ref{w2}) that each sum
\[
\sum_{e_{x}\in
J^y_{i}(X)}c^{xy}e_{x}\in\mathcal{A}_{i}(X)\ \ \   \mbox{and}\ \  \sum_{e_{\tilde{y}}\in
J^x_{n-i}(Y)}e_{\tilde{y}}\in\mathcal{A}_{n-i}(Y). \tag{5.8}\label{cons}
\]
For the sake of simplicity, given any elementary $n$-path $e_{v}=e_{0\cdots n}\in P_{n}(V)$
and integer $0\leq l\leq n$, we denote $e_{v(\widehat{l})}=e_{0\cdots
(l-1)\widehat{l}(l+1)\cdots n}$. On the other hand by (\ref{wori}) we can compute as follows:

\begin{align*}
\partial(w)&=\sum_{e_{y}\in P_{n-s}(Y)}\bigg(\sum_{e_{x}\in J^y_{s}(X)}c^{xy}\partial(e_{x})\bigg)\times e_{y}\\
&\quad+\sum_{J^y_{s}(X),J^x_{n-s}(Y)}\Bigg(\sum_{e_{x}\in J^y_{s}(X)}\bigg(e_{x}\times(-1)^sc^{xy}\sum_{e_{\tilde{y}}\in J^x_{n-s}(Y)}\partial(e_{\tilde{y}})\bigg)\Bigg)\\
\end{align*}
\begin{align*}
\quad\quad\quad\quad\quad\quad\quad&\quad+\sum_{J^{y^{\prime}}_{s+1}(X),J^{x^\prime}_{n-s-1}(Y)}\Bigg(\sum_{e_{x^\prime}\in J^{y^{\prime}}_{s+1}(X)}\bigg(\partial(e_{x^\prime})\times c^{x^\prime y^{\prime}}\sum_{e_{\tilde{y}^{\prime}}\in J^{x^\prime}_{n-s-1}(Y)}e_{\tilde{y}^{\prime}}\bigg)\Bigg)+P(Z)\\
&=\sum_{e_{y}\in P_{n-s}(Y)}\bigg(\sum_{e_{x}\in J^y_{s}(X)}c^{xy}\partial(e_{x})\bigg)\times e_{y}\tag{5.9}\label{pw}\\
&\quad+\sum_{J^y_{s}(X),J^x_{n-s}(Y)}\Bigg(\sum_{e_{x}\in J^y_{s}(X)}\bigg(e_{x}\times\bigg((-1)^sc^{xy}\sum_{e_{\tilde{y}}\in J^x_{n-s}(Y)}\partial(e_{\tilde{y}})\\
&\quad\quad+(-1)^lc^{x^\prime y^{\prime}}\sum_{\stackrel{e_{x^\prime(\widehat{l})}=e_x}{e_{\tilde{y}^{\prime}}\in J^{x^\prime}_{n-s-1}(Y)}}e_{\tilde{y}^{\prime}}\bigg)\bigg)\Bigg)\\
&\quad+\sum_{J^{y^{\prime}}_{s+1}(X),J^{x^\prime}_{n-s-1}(Y)}\Bigg(\sum_{\stackrel{e_{x^\prime(\widehat{l^\prime})}\notin\{e_x\}}{e_{x^\prime}\in J^{y^{\prime}}_{s+1}(X)}}\bigg((-1)^{l^\prime}e_{x^\prime(\widehat{l^\prime})}\times c^{x^\prime y^{\prime}}\sum_{e_{\tilde{y}^{\prime}}\in J^{x^\prime}_{n-s-1}(Y)}e_{\tilde{y}^{\prime}}\bigg)\Bigg)+P(Z)
\end{align*}
where $P(Z)$ denotes the rest summand of the form $\sum_{e_x\in
P_{r}(X),e_y\in P_{t}(Y)}c^{xy}e_x\times e_y$ with $s<r\leq n-1$ and
$t=n-1-r$. Thus it follows from $\partial(w)\in P_{n-1}(Z)$, (\ref{pw}) and
Lemma \ref{lem5.2} that
$$\sum_{e_{x}\in J^y_{s}(X)}c^{xy}\partial(e_{x})\in\mathcal{A}_{s-1}(X)$$
and
$$(-1)^sc^{xy}\sum_{e_{\tilde{y}}\in J^x_{n-s}(Y)}\partial(e_{\tilde{y}})
+(-1)^lc^{x^\prime y^{\prime}}\sum_{\stackrel{e_{x^\prime(\widehat{l})}=e_x}{e_{\tilde{y}^{\prime}}\in J^{x^\prime}_{n-s-1}(Y)}}e_{\tilde{y}^{\prime}}\in\mathcal{A}_{n-s-1}(Y),$$
or equivalently
$$\sum_{e_{\tilde{y}}\in J^x_{n-s}(Y)}\partial(e_{\tilde{y}})\in\mathcal{A}_{n-s-1}(Y).$$
Namely, by (\ref{cons}) and Lemma \ref{lem5.1} one has
\[
P=\bigg(\sum_{e_{x}\in
J^y_{s}(X)}c^{xy}e_{x}\bigg)\times\bigg(\sum_{e_{\tilde{y}}\in
J^x_{n-s}(Y)}e_{\tilde{y}}\bigg)\in\Omega_{s}(X)\times\Omega_{n-s}(Y)\subseteq\Omega_n(Z).
\]
Now if $w=P$ then we are finished. Otherwise one can rewrite
$w=P+Q$ where $Q$ denotes the remaining summands in
(\ref{w2}). Note that $w\in\Omega_n(Z)$ and $P\in\Omega_n(Z)$ imply
$Q\in\Omega_n(Z)$. Also one observes that both of the total
numbers of summands $c^{xy}e_x\times e_y$ consisting of $P$ and $Q$ are
less than $b$. Thus by inductive hypothesis one has the desired result.
\end{proof}

By Definition \ref{admiss} and Lemma \ref{keylem5} one immediately has the following result.
\begin{cor}\label{cor5.5}
All the paths $e_{x}\in P(X)$ and $e_{y}\in P(Y)$ in Lemma \ref{lem5.3} are admissible.
\end{cor}

This corollary allows us to define an $R$-linear map
$$F^{Z}_{S}: \Omega_{k}(Z)\rightarrow C_{k}(S(X)\times S(Y); R)$$
by the formula
$$F^{Z}_{S}(e_{x}\times e_{y})=\phi_{e_x}\times\phi_{e_y}$$
where $e_{x}$ and $e_{y}$ are admissible elementary paths in $P_{l}(X)$
and $P_{k-l}(Y)$, respectively, and the symbol `$\times$' means different
cross products on the two sides of the equation.

\begin{lem}\label{lem5.6}
The map $F^Z_{S}: \Omega_{k}(Z)\rightarrow C_{k}(S(X)\times S(Y); R)$ is injective.
\end{lem}

\begin{proof}
Let $m+n=k$ and suppose that all the elementary paths in this proof are admissible, we need to show that for any non-zero $(m+n)$-path
$w\in\Omega_{m+n}(Z)$ one always has $F^Z_{S}(w)\neq0$. By Lemma \ref{keylem5} one can write
\begin{align*}
w=\sum^{m+n=k}_{e_{x}\in P_m(X),\ e_{y}\in P_n(Y)}c^{xy}(e_{x}\times e_{y}).
\end{align*}
We claim that $F^Z_{S}(w)=0$ implies all coefficients $c^{xy}=0$. In fact, replace $f^j_m$
and $g^j_n$ by $\phi_{e_{x}}$ and $\phi_{e_{y}}$, respectively, as in the
proof of Lemma \ref{lem5.1} one has
$$F^Z_{S}(w)=\sum_{e_{\alpha}\in P_m(X),\ e_{\beta}\in P_n(Y)}^{m+n=k}c^{xy}\bigg(\sum_{\delta}(-1)^{|\delta|}(\phi_{e_{x}}\times\phi_{e_{y}})l_{\delta}\bigg).$$
Note that all $l_\delta(\Delta^{m+n}_{\sigma})$ are different since all $e_x$ and $e_y$
are different from each other. Thus each map $(\phi_{e_{x}}\times\phi_{e_{y}})l_{\delta}$
as the characteristic map $\Delta^{m+n}_{\sigma}\rightarrow S(X)\times S(Y)$
has different images from each other since a characteristic map is a
homeomorphism on the interior of some standard simplex. That is to say,
$F^Z_{S}(w)=0$ if and only if all coefficients $c^{xy}=0$, as asserted.
\end{proof}

\begin{prop}\label{prop5.7}
The chain complex $\Omega_{\ast}(Z)$ can be viewed as a subcomplex of the
simplicial chain complex $C_{\ast}(S(X)\times S(Y); R)$ via the map $F^Z_{S}$.
\end{prop}

\begin{proof}
By Lemma \ref{lem5.6} it suffices to verify that the map $F^Z_{S}$ is a chain map.
Let $$w=\sum_{e_{x}\in P_{i}(X),\ e_{y}\in P_{n-i}(Y)}c^{xy}(e_{x}\times e_{y})$$
be an $n$-path in $\Omega_{n}(Z)$. One computes that
\begin{align*}
F_S^Z\big(\partial(w)\big)&=F_S^Z\bigg(\sum_{e_{x}\in P_{i}(X),\ e_{y}\in P_{n-i}(Y)}c^{xy}\big(\partial(e_{x})\times e_{y}+(-1)^ie_x\times\partial(e_y)\big)\bigg)\\
&=\sum_{e_{x}\in P_{i}(X),\ e_{y}\in P_{n-i}(Y)}c^{xy}\big(\phi_{\partial(e_{x})}\times \phi_{e_{y}}+(-1)^i\phi_{e_x}\times\phi_{\partial(e_y)}\big)\\
&=\sum_{e_{x}\in P_{i}(X),\ e_{y}\in P_{n-i}(Y)}c^{xy}\partial(\phi_{e_x}\times\phi_{e_y})\\
&=\partial\big(F_S^Z(w)\big)
\end{align*}
where the third equality follows by (\ref{ptimp}). It then finishes the proof.
\end{proof}

We are ready to prove the first main result in this section:
\begin{mytheo}\label{mainthm1}
Let $P(X)$ and $P(Y)$ be two regular path complexes and $R$ a commutative ring. Then for
their Cartesian product $P(Z)=P(X)\boxplus P(Y)$ the following isomorphism of
chain complexes holds:
\begin{align*}
\Omega_{\ast}(X)\otimes_{R}\Omega_{\ast}(Y)\cong \Omega_{\ast}(Z) \tag{5.10}\label{iso5.10}
\end{align*}
whose mapping is given by $u\otimes v \mapsto u\times v$.
\end{mytheo}

\begin{proof}
Let us first inspect the map $$F_S\otimes F_S: \Omega_{\ast}(X)\otimes_{R}\Omega_{\ast}(Y)\rightarrow C_{\ast}(X; R)\otimes_{R} C_{\ast}(Y; R)$$
defined by the formula
$$(F_S\otimes F_S)(u\otimes v)=F_S(u)\otimes F_S(v)$$
for any $u\in\Omega_{m}(X)$ and $v\in\Omega_{n}(Y)$. By Lemma \ref{lem3.7}, as in the proof of Proposition \ref{prop5.7}, on can easily verify that
$F_S\otimes F_S$ is a chain map. Furthermore we shall show that
$F_S\otimes F_S$ is injective. Suppose $u=\sum_{e_x\in P_{m}(X)}k^xe_x\in\Omega_{m}(X)$, $v=\sum_{e_y\in P_{n}(Y)}k^ye_y\in\Omega_{n}(Y)$ and $(F_S\otimes F_S)(u\otimes v)=0$, then
we have
$$u\otimes v=\sum_{e_{x}\in P_{m}(X),\ e_{y}\in P_{n}(Y)}k^{x}k^y(e_{x}\otimes e_{y}).$$
and it follows from
$$(F_S\otimes F_S)(u\otimes v)=\sum_{e_{x}\in P_{m}(X),\ e_{y}\in P_{n}(Y)}k^{x}k^y(\phi_{e_{x}}\otimes\phi_{e_{y}})=0$$
that all coefficients $k^{x}k^y=0$, since all $\phi_{e_{x}}$ and $\phi_{e_{y}}$ lay respectively in the bases of $C_{m}(X; R)$ and
$C_{n}(Y; R)$ gives that all $\phi_{e_{x}}\otimes\phi_{e_{y}}$ lay in the basis of $C_{m}(X; R)\otimes C_{n}(Y; R)$. It follows that $u\otimes v=0$, as desired.

Now we obtain a diagram
\[
\Omega_{\ast}(X)\otimes_{R}\Omega_{\ast}(Y)\stackrel{F_{\Omega}}{\longrightarrow}C_{\ast}(S(X)\times S(Y); R)\stackrel{F_S^Z}{\longleftarrow}\Omega_{\ast}(Z)
\]
where $F_{\Omega}=G\circ(F_S\otimes F_S)$ and $G$ is given by Lemma \ref{lem5.1}, it is obvious that $F_{\Omega}$ is injective. We claim that
Im$F_{\Omega}$=Im$F_S^Z$. Indeed, for any $w\in\Omega_\ast(Z)$, by Lemma \ref{keylem5} one can write $w=\sum_{p_{x}, q_{y}}(p_{x}\times q_{y})$
where $p_x$ goes through a subset $I\subset\Omega_\ast(X)$ and $q_y$ goes through a subset $J\subset\Omega_\ast(Y)$, let $w^\prime=\sum_{p_{x},
q_{y}}(p_{x}\otimes q_{y})$ one immediately has $F_{\Omega}(w^\prime)=F_S^Z(w)$, that is, Im$F_S^Z\subseteq$ Im$F_{\Omega}$. The inverse
inclusion is obvious by Lemma \ref{lem5.2}. Thus the map $(F^Z_S)^{-1}\circ F_{\Omega}$ has meaning and it gives the desired isomorphism
(\ref{iso5.10}) by sending $u\otimes v$ to $u\times v$.
\end{proof}

K\"{u}nneth formula is used to compute the (co)homology of a product space
in terms of the (co)homology of the factors. For path complexes over a field,
a type of K\"{u}nneth formula also holds (see, \cite[Theorem 7.4]{Grig1}).
In fact, for any principle ideal domain $R$, we have the following more general result.

\begin{cor}\label{cor5.9}
Let $P(X)$ and $P(Y)$ be two regular path complexes and $R$ a PID. Then, for each $n$, there holds a K\"{u}nneth formula by the following natural splitting
short exact sequence
$$0\rightarrow\oplus_{i}\big(\mbox{H}_{i}(X)\otimes_{R}\mbox{H}_{n-i}(Y)\big)\rightarrow
\mbox{H}_{n}(Z)\rightarrow \oplus_{i}\mbox{Tor}_1^{R}\big(\mbox{H}_{i}(X), \mbox{H}_{n-i-1}(Y)\big)\rightarrow 0.$$
\end{cor}

\begin{proof}
By Theorem \ref{mainthm1} one has
$$\mbox{H}_{n}(Z)=\mbox{H}_{n}\big(\Omega_{\ast}(Z)\big)\cong\mbox{H}_{n}\big(\Omega_{\ast}(X)\otimes_{R}\Omega_{\ast}(Y)\big).$$
Note that each $\mathcal{A}_{n}(X)$ $(n\geq 0)$ is a finitely generated
free $R$-module, thus $\Omega_{n}(X)$ is also a free $R$-module since
$R$ is a PID. Therefore, the result follows from \cite[Theorem 3B.5]{Hatc}.
\end{proof}

\subsection{ The case of join }
\indent

\noindent Comparing with the approach of by considering the Cartesian product, there is another way to derive the K\"{u}nneth formula
via an operation called the \emph{join} (see Definition \ref{def5.10} below) of two regular path complexes. For path complexes over
a field $K$, this is exactly what \cite[Theorem 6.5]{Grig1} says. In fact the general result also holds when one replaces the field $K$ by
any commutative ring $R$. Before we set off to prove this, let us do some preparation.

\begin{defn}[\cite{Grig1},Definition 6.1]\label{def5.10}
\emph{Given two disjoint finite sets $X$, $Y$ and their path complexes $P(X)$ , $P(Y)$, set
$Z=X\times Y$ and define a path complex $P(Z)$ as follows: $P(Z)$ consists of all paths of the form
$uv$ where $u\in P(X)$ and $v\in P(Y)$. The path complex $P(Z)$ is called a \emph{join} of $P(X)$, $P(Y)$
and is denoted by $P(Z)=P(X)\ast P(Y)$.}
\end{defn}

Given any two paths $u\in P_{i-1}(X)$ and $v\in P_{n-i}(Y)$, it is easy to check that, the differential operator acting on the join $uv\in
P_n(Z)$ is give by the following formula:
$$\partial(uv)=\partial(u)v+(-1)^{i}u\partial(v).$$
For more properties and examples of the join of two regular path complexes the reader may refer to \cite{Grig1}. To prove the asserted K\"{u}nneth formula, we proceed
by a parallel way as that of proving Theorem \ref{mainthm1}. The key idea is to simply replace the symbol ``$\times$" of Cartesian product by
the symbol ``$\ast$" of join in the previous proofs, and treat carefully the corresponding dimensions.

In details, suppose we are given two
regular path complexes $P(X)$ and $P(Y)$, denote $P(Z)$ their join. We see that by definition
$\Omega_{s-1}(X)\ast\Omega_{k-s}(Y)\subseteq\Omega_{k}(Z)$ (comparing with Lemma \ref{lem5.2}) and each path $w\in\Omega_n(Z)$ admits a
representation
$$w=\sum_{i=1}^{n}\sum_{e_{x}\in P_{i-1}(X),\ e_{y}\in P_{n-i}(Y)}c^{xy}(e_{x}\ast e_{y})$$
with finitely nonzero coefficients $c^{xy}\in R$, which are uniquely determined by $w$ since obviously $e_{xy}=e_{x}\ast e_{y}$ across all
$e_{x}\in P(X)$ and $e_{y}\in P(Y)$ are $R$-linearly independent (comparing with Lemma \ref{lem5.3}). Now for the chain complex
$\Omega_{\ast}(X)$, we consider a new chain complex $\Omega^{\prime}_{\ast}(X)$ which is defined by the formula
$\Omega^{\prime}_{i}(X)=\Omega_{i-1}(X)$ and $\partial^\prime_i=\partial_{i-1}$. With this trick one is able to prove the following result.

\begin{mytheo}
Let $P(X)$ and $P(Y)$ be two regular path complexes and $R$ a commutative ring. Then for
their join $P(Z)=P(X)\ast P(Y)$ the following isomorphism of chain complexes holds:
\begin{align*}
\Omega_{\ast}(Z)\cong\Omega^{\prime}_{\ast}(X)\otimes_{R}\Omega_{\ast}(Y) \tag{5.11}\label{iso5.11}
\end{align*}
whose mapping is given by $u\otimes v \mapsto u\ast v$.

If furthermore $R$ is a PID. Then there holds a K\"{u}nneth formula by the following natural splitting short
exact sequence
\[
0\rightarrow\oplus_{i}(\mbox{H}_{i-1}(X)\otimes_{R}\mbox{H}_{n-i}(Y))\rightarrow
\mbox{H}_{n}(Z)\rightarrow \oplus_{i}\mbox{Tor}_1^{R}(\mbox{H}_{i-1}(X), \mbox{H}_{n-i-1}(Y))\rightarrow 0 \tag{5.12}\label{iso7}
\]
for each $n$.
\end{mytheo}

\begin{proof}
First note that as in Corollary \ref{cor5.9}, (\ref{iso7}) follows directly from
(\ref{iso5.11}). To obtain the isomorphism (\ref{iso5.11}), let us consider the map
$$F: \Omega^{\prime}_{\ast}(X)\otimes \Omega_{\ast}(Y)\rightarrow \Omega_{\ast}(Z)$$
given by $u\otimes v \mapsto u\ast v$ for any $u\in\Omega^{\prime}_{i}(X)$ and $v\in\Omega_{n-i}(Y)$.
One sees that the basis of  $\Omega^{\prime}_{i}(X)\otimes\Omega_{n-i}(Y)$ consists of all elements
of the form $e_x\otimes e_y$ where $e_x$ and $e_y$ are some elementary paths in $\Omega_{i-1}(X)$ and $\Omega_{n-i}(Y)$.
Apparently $F$ is injective since all $e_{xy}=e_x\ast e_y$ are $R$-linearly
independent.

Now we are done if we can show that the map $F$ is surjective, but this follows directly from the proof of Lemma \ref{keylem5}. To see this one
needs only to replace the symbol ``$\times$" by ``$\ast$", $\Omega_i(X)$ $P_i(X)$ and similarly $J_i(X)$ etc. by $\Omega^\prime_i(X)$,
$P^\prime_{i}(X)$ and $J^\prime_{i}(X)$ etc., respectively, where $P^\prime_{i}(X):=P_{i-1}(X)$ and $J^\prime_{i}(X):=J_{i-1}(X)$. Then the
proof still validates and this gives that $F$ is surjective.
\end{proof}

\begin{rem} If $R=K$ is a field, one immediately obtains \cite[Theorems 6.5 and 7.6]{Grig1} by Theorem \ref{mainthm1}, Corollary
\ref{cor5.9} and Theorem 5.11. But note that the proofs of \cite[Theorems 6.5 and 7.6]{Grig1} used the theory of vector spaces over a field,
which no more works for $R$-modules \textbf{when $R$ is a commutative ring}, so our proofs of Theorems 5.8, 5.11 and the proof of \cite[Theorems 6.5 and
7.6]{Grig1} have the different essentially points.
\end{rem}

{\bf Acknowledgements:}\;
This first author is supported by the National Natural Science Foundation of China(No.12071422) and the Zhejiang Provincial Natural Science Foundation of China (No.LY19A010023), the second author is supported by the Youth
Program of Provincial Natural Science Foundation of Zhejiang (No. LQ20A010008).

\end{document}